\documentclass{amsart}

 \usepackage{amssymb,latexsym,amsmath,amsthm}

 \usepackage{fullpage} 
 \renewcommand{\div}{\mathop{\mathrm{div}}\nolimits}

\usepackage{xypic}
\usepackage[all]{xy}

\usepackage{url}

\numberwithin{equation}{section}

\usepackage[numbers,sort&compress]{natbib}

\newtheorem{thm}{Theorem}[section]

\newtheorem{dfn}{Definition}[section]

\newtheorem{lemma}{Lemma}[section]

\newtheorem{remark}{Remark}[section]

\newtheorem{cor}{Corollary}[section]

\def\eps{\varepsilon}

\def\R{{\mathbb R} }

\def\r{\R^{n+1}_{+}}

\def\br{\partial\r}

\begin{document}
\author{Mostafa Fazly}

\address{Department of Mathematical and Statistical Sciences, CAB 632, University of Alberta, Edmonton, Alberta, Canada T6G 2G1}

\email{fazly@ualberta.ca}

\author{Juncheng Wei}
\address{Department of Mathematics, University of British Columbia, Vancouver, B.C. Canada V6T 1Z2.}
\email{jcwei@math.ubc.ca}

\thanks{Both authors are supported by  Natural Sciences and Engineering Research Council of Canada (NSERC) grants. We thank Pacific Institute for the Mathematical Sciences (PIMS) for hospitality.}

\title{On finite Morse index solutions of higher order fractional Lane-Emden equations}
\maketitle

\vspace{3mm}

\begin{abstract} We classify finite Morse index solutions of the following nonlocal Lane-Emden equation \begin{equation*}
(-\Delta)^{s} u=|u|^{p-1} u \ \ \ \mathbb{R}^n
\end{equation*}
for $1<s<2$ via a novel monotonicity formula. For local cases $s=1$ and $s=2$ this classification is provided by Farina in \cite{f} and Davila, Dupaigne, Wang and Wei in \cite{ddww}, respectively. Moreover, for the nonlocal case $0<s<1$ finite Morse index solutions are classified by  Davila, Dupaigne and Wei in \cite{ddw}.
\end{abstract}

\tableofcontents

\section{Introduction and Main Results}
We study the classification of stable solutions of the following equation
\begin{equation}\label{main}
(-\Delta)^{s} u=|u|^{p-1} u \ \ \ \mathbb{R}^n
\end{equation}
where  $(-\Delta)^{s}$  is the fractional Laplacian operator for $1<s<2$. For various parameters $s$ and $p$ this equation has been of attention of many experts in the field of partial differential equations.
\subsection{The local case}
For the case of $s=1$,  a celebrated result of Gidas and Spruck in \cite{gs} shows that the only nonnegative solution solution of the Lane-Emden equation is $u=0$ for $1<p<p_S$ where
\begin{eqnarray*}\label{ps}
p_S(n)= \left\{ \begin{array}{lcl}
\hfill \infty   \ \ && \text{if } n \le 2,\\
\hfill  \frac{n+2}{n-2} \ \ && \text{if } n > 2,
\end{array}\right.
\end{eqnarray*}
that is called the Sobolev exponent. In addition, for the critical case $p=p_S(n)$ it is shown by Caffarelli-Gidas-Spruck \cite{cgs} that there is a unique (up to translation and rescaling) positive solution for the Lane-Emden equation. For finite Mose index solutions (not necessarily positive), such classification is provided by Farina in \cite{f} and the critical exponent, called Joseph-Lundgren \cite{jl} exponent, is given by
 \begin{eqnarray}\label{pc1}
p_c(n)= \left\{ \begin{array}{lcl}
\hfill \infty   \ \ && \text{if } n \le 10,\\
\hfill  \frac{(n-2)^2 -4n+8\sqrt{n-1} }{(n-2)(n-10)} \ \ && \text{if } n \ge 11,
\end{array}\right.
\end{eqnarray}
Note that $p_c(n)>p_S(n)$ for $n>2$.

For the case of $s=2$, Wei and Xu \cite{wx} (see also Lin \cite{lin}) proved that the only nonnegative solution of the fourth order Lane-Emden equation is $u=0$ for $1<p<p_S$ where $p_S(n)$ is the Sobolev exponent, i.e.
\begin{eqnarray}\label{pc}
p_S(n)= \left\{ \begin{array}{lcl}
\hfill \infty   \ \ && \text{if } n \le 4,\\
\hfill  \frac{n+4}{n-4} \ \ && \text{if } n > 4.
\end{array}\right.
\end{eqnarray}
Moreover, for the critical case $p=p_S(n)$ they showed that there is a unique (up to translation and rescaling) positive solution for the fourth order Lane-Emden equation.  For finite Mose index solutions (not necessarily positive),  Davila, Dupaigne, Wang and Wei in \cite{ddww} gave a complete classification. The Joseph-Lundgren  exponent, computed by Gazzola and  Grunau in \cite{gg}, is the following
 \begin{eqnarray}\label{pc2}
p_c(n)= \left\{ \begin{array}{lcl}
\hfill \infty   \ \ && \text{if } n \le 12,\\
\hfill  \frac{ n+2 -\sqrt{ n^2 +4 -n\sqrt{n^2-8n+32} }}{n-6-\sqrt{n^2+4-n\sqrt{n^2-8n+32}}} \ \ && \text{if } n \ge 13,
\end{array}\right.
\end{eqnarray}
The key idea of the proof of Davila, Dupaigne, Wang and Wei in \cite{ddww}  is proving and applying a  monotonicity formula. Note that a monotonicity formula for the second order equation is established by F. Pacard in \cite{pak}.

We also refer the interested readers to Wei-Xu in \cite{wx} for classification of solutions of higher order conformally invariant equations, i.e. $s$ any positive integer.

\subsection{The nonlocal case}
Assume that $u\in C^{2\sigma}(\mathbb R^n)$, $\sigma >s>0$ and
$$ \int_{\mathbb R^n} \frac{|u(y)|}{(1+|y|)^{n+2s}} dy<\infty$$ so the fractional Laplacian of $u$
\begin{equation}
(-\Delta)^s u(x):= p.v. \int_{\mathbb R^n} \frac{u(x)-u(y)}{|x-y|^{n+2s}} dy
\end{equation}
is well-defined for every $x\in\mathbb R^n$.

For the case of $0<s<1$, a counterpart of the classification results of Gidas-Spruck \cite{gs} and Caffarelli-Gidas-Spruck \cite{cgs} holds for the fractional Lane-Emden equation, see the work of Li \cite{li} and Chen-Li-Ou \cite{clo}. In this case, the Sobolev exponent is the following
\begin{eqnarray}\label{pns}
p_S(n,s)= \left\{ \begin{array}{lcl}
\hfill \infty   \ \ && \text{if } n \le 2s,\\
\hfill  \frac{n+2s}{n-2s} \ \ && \text{if } n > 2s.
\end{array}\right.
\end{eqnarray}
Very recently, for the case of $0<s<1$, Davila, Dupaigne and Wei \cite{ddw} gave a complete classification of finite Morse index solutions of (\ref{main}) via proving and applying a monotonicity formula.  As a matter of fact, they proved that for either $1<p<p_S(n, s)$ or $p>p_S(n, s)$ and
$$ p \frac{\Gamma(\frac{n}{2}-\frac{s}{p-1}) \Gamma(s+\frac{s}{p-1})}{\Gamma(\frac{s}{p-1}) \Gamma(\frac{n-2s}{2}-\frac{s}{p-1})}>\frac{ \Gamma(\frac{n+2s}{4})^2 }{\Gamma(\frac{n-2s}{4})^2}$$ the only finite More index solution is zero. In this work, we are interested in knowing whether such classification results
hold for finite Morse index solutions of (\ref{main}) when $1<s<2$.

There are different ways of defining the fractional operator $(-\Delta)^{s}$ where $1<s<2$, just like the case of $0<s<1$. Applying the Fourier transform one can define the fractional Laplacian by
$$\widehat{ (-\Delta)^{s}}u(\zeta)=|\zeta|^{2s} \hat u(\zeta)$$ or equivalently define this operator inductively by  $(-\Delta)^{s}=  (-\Delta)^{s-1} o (-\Delta)$, see \cite{rs}.   Recently, Yang in \cite{yang} gave a characterization of the fractional Laplacian $(-\Delta)^{s}$, where $s$ is  any positive, noninteger number as the Dirichlet-to-Neumann map for a function $u_e$ satisfying a  higher order elliptic equation in the upper half space with one extra spatial dimension. This is a generalization of the work of Caffarelli and Silvestre in \cite{cs} for the case of $0<s<1$. See also Case-Chang \cite{cc} and Chang-Gonzales \cite{cg}.

Throughout this note set $b:=3-2s$ and define the operator $$\Delta_b w:=\Delta w+\frac{b}{y} w_y=y^{-b} \div(y^b \nabla w)$$
for a function $w\in W^{2,2}(\mathbb R^{n+1},y^b)$.

\begin{thm} \cite{yang} Let $1<s<2$. For functions $u_e \in W^{2,2} (\mathbb R^{n+1}_+, y^b)$ satisfying the equation
$$\Delta^2_b u_e=0$$ on the upper half space for $(x, y) \in \mathbb R^{n}\times \mathbb R_+$ where
$y$ is the special direction, and the
boundary conditions
\begin{eqnarray*}
u_e(x,0)&=&f(x) \\
\lim_{y\to 0} y^{b}\partial_y{u_e(x,0)}&=& 0
\end{eqnarray*}
along $\{y=0\}$ where $f(x)$ is some function defined on $H^s(\mathbb R^n)$ we have the result that
$$ (-\Delta)^s f(x)= C_{n,s} \lim_{y\to 0} y^{b} \partial_y \Delta_b u_e (x,y) $$
Moreover, $$ \int_{\mathbb R^n} |\xi|^{2s} |\hat{u(\xi)}|^2 d\xi=C_{n,s} \int_{\mathbb {R}^{n+1}_+} y^b |\Delta _b u_e(x,y)|^2 dx dy$$
\end{thm}
Applying the above theorem to solutions of (\ref{main}) we conclude that the extended function $u_e(x,y)$ where $x=(x_1,\cdots,x_n)$ and $y\in\mathbb R^+$ satisfies
\begin{eqnarray}\label{maine}
 \left\{ \begin{array}{lcl}
\hfill \Delta^2_b u_e&=& 0   \ \ \text{in}\ \ \mathbb{R}^{n+1}_{+},\\
\hfill  \lim_{y\to 0} y^{b}\partial_y{u_e}&=& 0   \ \ \text{in}\ \ \partial\mathbb{R}^{n+1}_{+},\\
\hfill \lim_{y\to 0} y^{b} \partial_y \Delta_b u_e &=& C_{n,s} |u|^{p-1} u  \ \ \text{in}\ \ \partial\mathbb{R}^{n+1}_{+}
\end{array}\right.
\end{eqnarray}
Moreover, $$ \int_{\mathbb R^n} |\xi|^{2s} |\hat{u(\xi)}|^2 d\xi=C_{n,s} \int_{\mathbb {R}^{n+1}_+} y^b |\Delta _b u_e(x,y)|^2 dx dy$$
Then $u(x)=u_e(x,0)$.  

 For $1<s<2$,  Chen et al in  \cite{cczy} have  classified all positive solutions of (\ref{main})   for $1<p \leq p_S (n, s)$.   The main goal of this paper is to classify all (positive or sign-changing) solutions of (\ref{main}) which are stable outside a compact set. To this end, we first  introduce the corresponding Joseph-Lungren's exponent.    As it is shown by Herbst in \cite{h} (and also \cite{ya}), for $n>2s$ the following Hardy inequality holds
$$\int_{\mathbb R^n} |\xi|^{2s} |\hat \phi|^2 d\xi > \Lambda_{n,s} \int_{\mathbb R^n} |x|^{-2s} \phi^2 dx$$
for any $\phi \in C_c^\infty(\mathbb R^n)$ where the optimal constant given by $$ \Lambda_{n,s}=2^{2s}\frac{ \Gamma(\frac{n+2s}{4})^2  }{ \Gamma(\frac{n-2s}{4})^2}$$

\begin{dfn}
We say that a solution $u$ of (\ref{main}) is stable outside a compact set  if there exists $R_0>0$ such that
\begin{equation}\label{stability} \int_{\mathbb R^n} \int_{\mathbb R^n} \frac{ ( \phi(x)-\phi(y) )^2 }{|x-y|^{n+2s}} dx dy-p \int_{\mathbb R^n} |u|^{p-1} \phi^2 \ge 0
\end{equation}
 for any $\phi\in C_c^\infty(\mathbb R^n\setminus \overline {B_{R_0}})$.
 \end{dfn}
In the following lemma we provide an explicit singular solution for  (\ref{main}).
\begin{lemma} Suppose that $1<s<2$ and $p>p_S(n, s)$ then
\begin{equation}
u_s(x) = A |x|^{-\frac{2s}{p-1}}
\end{equation}
where
$$
A^{p-1} =\frac{\Gamma(\frac{n}{2}-\frac{s}{p-1}) \Gamma(s+\frac{s}{p-1})}{\Gamma(\frac{s}{p-1}) \Gamma(\frac{n-2s}{2}-\frac{s}{p-1})}
$$
solves (\ref{main}).
\end{lemma}
\begin{proof}
From Lemma 3.1 in \cite{fall}, we conclude that when $0<t<1$, for any $-\frac{n+2t}{2}<\beta<\frac{n-2t}{2}$
\begin{equation}
(-\Delta)^{t} |x|^{\frac{2t-n}{2}+\beta} = \gamma_t(\beta) |x|^{\frac{2t-n}{2}+\beta-2t}
\end{equation}
where
\begin{equation}
\gamma_t(\beta) = 2^{2t} \frac{\Gamma(\frac{n+2t+2\beta}{4}) \Gamma(\frac{n+2t-2\beta}{4})}{\Gamma(\frac{n-2t-2\beta}{4})\Gamma(\frac{n-2t+2\beta}{4})}
\end{equation}
From the fact that $(-\Delta)^{s} =(-\Delta)o(-\Delta)^{t} $ for $0<t=s-1<1$ we have
\begin{eqnarray}\label{}
(-\Delta)^{s} |x|^{\frac{2t-n}{2}+\beta} =\gamma_t(\beta) (-\Delta) |x|^{\frac{2t-n}{2}+\beta-2t}=-\gamma_t(\beta)
 \eta_{t}(n+\eta_{t}-2) |x|^{\frac{2t-n}{2}+\beta-2t-2}
 \end{eqnarray}
where $\eta_t=\frac{2t-n}{2}+\beta-2t$. Now using the change of variable $t=s-1$ we get
\begin{eqnarray}\label{}
(-\Delta)^{s} |x|^{\frac{2s-n}{2}+\beta-1} &=&-\gamma_{s-1}(\beta)
 \eta_{s-1}(n+\eta_{s-1}-2) |x|^{\frac{2s-n}{2}+\beta-2s-1}
 \\&=& -\gamma_{s-1}(\beta)
 \eta_{s-1}(n+\eta_{s-1}-2)  \left( |x|^{\frac{2s-n}{2}+\beta-1}\right)^p
 \end{eqnarray}
where $\frac{\frac{2s-n}{2}+\beta-2s-1}{\frac{2s-n}{2}+\beta-1}=p$. From this we conclude that $\beta=\frac{-2s}{p-1}+\frac{n-2s}{2} +1$. This implies
\begin{equation}
u_s(x) = A |x|^{-\frac{2s}{p-1}}
\end{equation}
where
$$
A^{p-1} = \lambda\left(\frac{n-2s}{2} -\frac{2s}{p-1}+1\right)
$$
is a solution of (\ref{main}) for
\begin{equation}
\lambda(\beta) =   -\gamma_{s-1}(\beta)
 \eta_{s-1}(n+\eta_{s-1}-2)
\end{equation}
Elementary calculations show that
\begin{eqnarray}\label{al1}
\gamma_{s-1}(\beta) &=& 2^{2s-2} \frac{\Gamma(\frac{n}{2}- \frac{s}{p-1}) \Gamma(s+\frac{s}{p-1}-1)}{\Gamma(\frac{s}{p-1})\Gamma(\frac{n-2s}{2} -\frac{s}{p-1}+1)}
 \end{eqnarray}
and
\begin{eqnarray}\label{al2}
- \eta_{s-1}(n+\eta_{s-1}-2) = 4 \left(s+\frac{s}{p-1}-1\right)\left(\frac{n-2s}{2}-\frac{s}{p-1}\right)
  \end{eqnarray}
  From (\ref{al1}) and (\ref{al2}) and using the property $a\Gamma(a)=\Gamma(a+1)$ we conclude the desired result.

\end{proof}

Here is our main result.
\begin{thm}\label{mainthm}
Suppose that $n\ge 1$ and $1<s<\delta<2$. Let $u\in C^{2\delta}(\mathbb R^n)\cap L^1(\mathbb R^n, (1+|y|)^{n+2s}dy)$ be a solution of (\ref{main}) that is stable outside a compact set. Then either for $1<p<p_S(n,s)$ or for $p>p_S(n,s)$ and 
\begin{equation}\label{conditionp} 
p \frac{\Gamma(\frac{n}{2}-\frac{s}{p-1}) \Gamma(s+\frac{s}{p-1})}{\Gamma(\frac{s}{p-1}) \Gamma(\frac{n-2s}{2}-\frac{s}{p-1})}>\frac{ \Gamma(\frac{n+2s}{4})^2 }{\Gamma(\frac{n-2s}{4})^2}
\end{equation}
solution $u$ must be zero.  Moreover for the case $p=p_S(n,s)$, a solution $u$ has finite energy that is
$$ \int_{\mathbb R^n} |u|^{p+1} = \int_{\mathbb R^n} \int_{\mathbb R^n} \frac{(u(x)-u(y))^2}{|x-y|^{n+2s}}<\infty$$
If in addition $u$ is stable,  then $u$ must be zero.
\end{thm}
Note that when $s=1$ and $s=2$ assumption (\ref{conditionp})  is equivalent to $1<p<p_c(n)$ where $p_c(n)$ is given by (\ref{pc1}) and (\ref{pc2}), respectively. Here is the computation for the case of $s=1$. Note that when $s=1$ the assumption (\ref{conditionp}) is 
\begin{equation}\label{rrr} 
p \frac{\Gamma(\frac{n}{2}-\frac{1}{p-1}) \Gamma(1+\frac{1}{p-1})}{\Gamma(\frac{1}{p-1}) \Gamma(\frac{n}{2}-1-\frac{1}{p-1})}>\frac{ \Gamma(\frac{n+2}{4})^2 }{\Gamma(\frac{n-2}{4})^2}. 
\end{equation}
We now use  properties of the gamma function, i.g. $\Gamma(1+a)=a\Gamma(a)$ for $a>0$, to get
\begin{eqnarray}
\Gamma\left(\frac{n}{2}-\frac{1}{p-1}\right) &=& \left(\frac{n}{2}-1-\frac{1}{p-1}\right) \Gamma\left(\frac{n}{2}-1-\frac{1}{p-1}\right) \\
\Gamma\left(1+\frac{1}{p-1}\right) &=& \left(\frac{1}{p-1}\right)  \Gamma\left(\frac{1}{p-1}\right) \\
 \Gamma\left(\frac{n+2}{4}\right) &=&  \left(\frac{n-2}{4} \right) \Gamma\left(\frac{n-2}{4}\right).
  \end{eqnarray}
Substituting this in (\ref{rrr}) we get 
$$ p\left(\frac{n}{2}-1-\frac{1}{p-1}\right)\left(\frac{1}{p-1}\right) > \left(\frac{n-2}{4} \right)^2.$$ 
Straightforward calculations show that this is equivalent to $1<p<p_c(n)$ where $p_c(n)$ is given by (\ref{pc1}).

Some remarks are in order. Even though the proof of Theorem \ref{mainthm} follows from the general procedure used in \cite{ddww} and \cite{ddw}, there are a few new ingredients in our proofs. First (in Section 2) we have derived the monotonicity formula involving higher order fractional operators.  Second (in Section 3) we have developed a new and direct method to prove the non-existence of stable homogeneous solutions. This method avoids multiplication or integration by parts and works for any fractional operator.

The monotonicity formula we derived in Section 2 implicitly used the Pohozaev's type identity. For higher order factional operator the Pohozaev identity has  been derived recently  by Ros-Oton and Serra \cite{rs}.

\section{Monotonicity Formula}
The key technique of our proof is a monotonicity formula that is developed in this section.   Define
\begin{eqnarray*}\label{energy}
E(r,x,u_e)& :=& r^{2s\frac{p+1}{p-1}-n} \left(   \int_{  \mathbb{R}^{n+1}_{+}\cap B_r(x_0)} \frac{1}{2} y^{3-2s}|\Delta_b u_e|^2-  \frac{C_{n,s}}{p+1} \int_{  \partial\mathbb{R}^{n+1}_{+}\cap B_r(x_0)} u_e^{p+1}   \right)\\
&&- \frac{s}{p-1} \left( \frac{p+2s-1}{p-1}-n\right) r^{-3+2s+\frac{4s}{p-1}-n}  \int_{  \mathbb{R}^{n+1}_{+}\cap \partial B_r(x_0)} y^{3-2s} u_e^2 \\
&&-\frac{s}{p-1} \left( \frac{p+2s-1}{p-1}-n\right) \frac{d}{dr} \left[ r^{\frac{4s}{p-1}+2s-2-n}  \int_{  \mathbb{R}^{n+1}_{+}\cap \partial B_r(x_0)} y^{3-2s} u_e^2 \right]\\
&&+ \frac{1}{2} r^3   \frac{d}{dr} \left[ r^{\frac{4s}{p-1}+2s-3-n}  \int_{  \mathbb{R}^{n+1}_{+}\cap \partial B_r(x_0)} y^{3-2s} \left(  \frac{2s}{p-1} r^{-1} u+ \frac{\partial u_e}{\partial r}\right)^2 \right]\\
&&+ \frac{1}{2}    \frac{d}{dr} \left[ r^{2s\frac{p+1}{p-1}-n}  \int_{  \mathbb{R}^{n+1}_{+}\cap \partial B_r(x_0)} y^{3-2s}\left(  | \nabla u_e|^2 - \left|\frac{\partial u_e}{\partial r}\right|^2 \right) \right]\\
&&+ \frac{1}{2}    r^{2s\frac{p+1}{p-1}-n-1}  \int_{  \mathbb{R}^{n+1}_{+}\cap \partial B_r(x_0)} y^{3-2s} \left(  | \nabla u_e|^2 - \left|\frac{\partial u_e}{\partial r}\right|^2 \right)
\end{eqnarray*}

\begin{thm}\label{mono}
Assume that $n>\frac{p+4s-1}{p+2s-1}+ \frac{2s}{p-1}-b$. Then, $E(\lambda,x,u_e)$ is a nondecreasing function of $\lambda>0$. Furthermore,
\begin{equation}
\frac{dE(\lambda,x,u_e)}{d\lambda} \ge C(n,s,p) \  \lambda^{\frac{4s}{p-1}+2s-2-n}    \int_{  \mathbb{R}^{n+1}_{+}\cap \partial B_\lambda(x_0)}  y^{3-2s}\left(  \frac{2s}{p-1} r^{-1} u+ \frac{\partial u_e}{\partial r}\right)^2
\end{equation}
where $C(n,s,p)$ is independent from $\lambda$.
\end{thm}

\noindent{\bf Proof:}  Suppose that $x_0=0$ and the balls $B_\lambda$ are centred at zero. Set,
\begin{equation}
\bar E(u_e,\lambda):= \lambda^{2s\frac{p+1}{p-1}-n} \left(   \int_{  \mathbb{R}^{n+1}_{+}\cap B_\lambda} \frac{1}{2} y^{b} |\Delta_b u_e|^2 dx dy -  \frac{C(n,s)}{p+1} \int_{  \partial\mathbb{R}^{n+1}_{+}\cap B_\lambda} u_e^{p+1}   \right)
\end{equation}
 Define $v_e:=\Delta_b u_e$, $u_e^\lambda(X):=\lambda^{\frac{2s}{p-1}} u_e(\lambda X)$, and $v_e^\lambda(X):=\lambda^{\frac{2s}{p-1}+2} v_e(\lambda X)$ where $X=(x,y)\in\mathbb{R}^{n+1}_+$.  Therefore, $\Delta_b u_e^\lambda(X)=v_e^\lambda(X)$ and
 \begin{eqnarray}\label{mainex}
 \left\{ \begin{array}{lcl}
\hfill \Delta_b v^\lambda_e&=& 0   \ \ \text{in}\ \ \mathbb{R}^{n+1}_{+},\\
\hfill  \lim_{y\to 0} y^{b}\partial_y{u^\lambda_e}&=& 0   \ \ \text{in}\ \ \partial\mathbb{R}^{n+1}_{+},\\
\hfill \lim_{y\to 0} y^{b} \partial_y v_e^\lambda &=& C_{n,s} {(u^\lambda_e)}^p  \ \ \text{in}\ \ \partial\mathbb{R}^{n+1}_{+}
\end{array}\right.
\end{eqnarray}
In addition,  differentiating with respect to $\lambda$ we have
 \begin{equation}\label{uvl}
 \Delta_b \frac{du_e^\lambda}{d\lambda}=\frac{dv_e^\lambda}{d\lambda}.
 \end{equation}
Note that $$\bar E(u_e,\lambda)=\bar E(u_e^\lambda,1)= \int_{  \mathbb{R}^{n+1}_{+}\cap B_1} \frac{1}{2} y^b  (v_e^\lambda)^2 dx dy -  \frac{C_{n,s}}{p+1} \int_{  \partial\mathbb{R}^{n+1}_{+}\cap B_1} |u_e^\lambda|^{p+1}   $$
Taking derivate of the energy with respect to $\lambda$, we have
 \begin{eqnarray}\label{energy}
\frac{d\bar E(u_e^\lambda,1)}{d\lambda}= \int_{  \mathbb{R}^{n+1}_{+}\cap B_1} y^b v_e^\lambda \frac{dv_e^\lambda}{d\lambda}\  dx dy -  C_{n,s} \int_{  \partial\mathbb{R}^{n+1}_{+}\cap B_1} |u_e^\lambda|^{p}  \frac{du_e^\lambda}{d\lambda}
\end{eqnarray}
Using (\ref{mainex}) we end up with
\begin{eqnarray}\label{energy}
\frac{d\bar E(u_e^\lambda,1)}{d\lambda}= \int_{  \mathbb{R}^{n+1}_{+}\cap B_1}  y^b v_e^\lambda \frac{dv_e^\lambda}{d\lambda}\  dx dy -  \int_{  \partial\mathbb{R}^{n+1}_{+}\cap B_1}  \lim_{y\to 0} y^{b}  \partial_y v_e^\lambda   \frac{du_e^\lambda}{d\lambda}
\end{eqnarray}
From (\ref{uvl}) and by  integration by parts we have
 \begin{eqnarray*}
 \int_{  \mathbb{R}^{n+1}_{+}\cap B_1} y^b v_e^\lambda \frac{dv_e^\lambda}{d\lambda}
&=&  \int_{  \mathbb{R}^{n+1}_{+}\cap B_1} y^b \Delta_b u_e^\lambda \Delta_b \frac{du_e^\lambda}{d\lambda} \\&=&  - \int_{  \mathbb{R}^{n+1}_{+}\cap B_1} \nabla \Delta_b u^\lambda_e \cdot \nabla \left(\frac{du^\lambda_e}{d \lambda}\right) y^b + \int_{ \partial( \mathbb{R}^{n+1}_{+}\cap B_1)} \Delta_b u_e^\lambda y^b \partial_{\nu} \left(   \frac{du^\lambda_e}{d \lambda}  \right)
\end{eqnarray*}
Note that
 \begin{eqnarray*}
-\int_{  \mathbb{R}^{n+1}_{+}\cap B_1} \nabla \Delta_b u_e\cdot \nabla \frac{du^\lambda_e}{d \lambda} y^b &=& \int_{  \mathbb{R}^{n+1}_{+}\cap B_1} \div( \nabla\Delta_b u^\lambda_e y^b) \frac{du^\lambda_e}{d \lambda} -  \int_{ \partial( \mathbb{R}^{n+1}_{+}\cap B_1)} y^b   \partial_\nu (\Delta_b u^\lambda_e )  \frac{du^\lambda_e}{d \lambda}
\\& =&\int_{  \mathbb{R}^{n+1}_{+}\cap B_1} y^b  \Delta_b^2 u^\lambda_e  \frac{du^\lambda_e}{d \lambda} -  \int_{ \partial( \mathbb{R}^{n+1}_{+}\cap B_1)} y^b   \partial_\nu (\Delta_b u^\lambda_e )  \frac{du^\lambda_e}{d \lambda}
\\& =&-  \int_{ \partial( \mathbb{R}^{n+1}_{+}\cap B_1)} y^b   \partial_\nu (\Delta_b u^\lambda_e )  \frac{du^\lambda_e}{d \lambda}
\end{eqnarray*}
Therefore,
 \begin{eqnarray*}
 \int_{  \mathbb{R}^{n+1}_{+}\cap B_1} y^b v_e^\lambda \frac{dv_e^\lambda}{d\lambda}
&=& \int_{ \partial( \mathbb{R}^{n+1}_{+}\cap B_1)} \Delta_b u_e^\lambda y^b \partial_{\nu} \left(   \frac{du^\lambda_e}{d \lambda}  \right) -  \int_{ \partial( \mathbb{R}^{n+1}_{+}\cap B_1)} y^b   \partial_\nu (\Delta_b u^\lambda_e )  \frac{du^\lambda_e}{d \lambda}
\end{eqnarray*}
Boundary of $\mathbb{R}^{n+1}_{+}\cap B_1$ consists of   $\partial\mathbb{R}^{n+1}_{+}\cap B_1$ and  $\mathbb{R}^{n+1}_{+}\cap \partial B_1$. Therefore,
\begin{eqnarray*}
 \int_{  \mathbb{R}^{n+1}_{+}\cap B_1} y^b v_e^\lambda \frac{dv_e^\lambda}{d\lambda}
&=&  \int_{  \partial\mathbb{R}^{n+1}_{+}\cap B_1} - v_e^\lambda \lim_{y\to 0}  y^b \partial_y\left (\frac{du_e^\lambda}{d\lambda}\right) +   \lim_{y\to 0} y^b \partial_y v_e^\lambda \frac{du_e^\lambda}{d\lambda} \\&&+ \int_{  \mathbb{R}^{n+1}_{+}\cap \partial B_1} y^b v_e^\lambda \partial_r \left (\frac{du_e^\lambda}{d\lambda}\right) -  y^b \partial_r v_e^\lambda \frac{du_e^\lambda}{d\lambda}
\end{eqnarray*}
where $r=|X|$, $X=(x,y)\in \mathbb{R}^{n+1}_+$ and $\partial_r=\nabla\cdot \frac{X}{r}$ is the corresponding radial derivative.   Note that the first integral in the right-hand side vanishes since $\partial_y\left (\frac{du_e^\lambda}{d\lambda}\right)=0$ on $ \partial\mathbb{R}^{n+1}_{+}$. From (\ref{energy}) we obtain
 \begin{eqnarray}\label{energy2}
\frac{d\bar E(u_e^\lambda,1)}{d\lambda}= \int_{  \mathbb{R}^{n+1}_{+}\cap \partial B_1}  y^b\left ( v_e^\lambda \partial_r \left (\frac{du_e^\lambda}{d\lambda}\right) -  \partial_r \left(v_e^\lambda\right) \frac{du_e^\lambda}{d\lambda} \right)
\end{eqnarray}
Now note that from the definition of $u_e^\lambda$ and $v_e^\lambda$ and by differentiating in $\lambda $ we get the following for $X\in\mathbb{R}^{n+1}_+$
\begin{eqnarray}\label{ulambda}
\frac{du_e^\lambda(X)}{d\lambda}&=&\frac{1}{\lambda} \left(  \frac{2s}{p-1} u_e^\lambda(X) +r \partial_r u_e^\lambda(X) \right)\\\label{vlambda}
\frac{dv_e^\lambda(X)}{d\lambda}&=&\frac{1}{\lambda} \left(  \frac{2(p+s-1)}{p-1} v_e^\lambda(X) +r \partial_r v_e^\lambda(X) \right)
\end{eqnarray}

Therefore, differentiating with respect to $\lambda $ we get
\begin{eqnarray*}
\lambda \frac{d^2u_e^\lambda(X)}{d\lambda^2} + \frac{du_e^\lambda(X)}{d\lambda}=  \frac{2s}{p-1} \frac{du_e^\lambda(X)}{d\lambda} +r \partial_r \frac{du_e^\lambda(X)}{d\lambda}
\end{eqnarray*}
So, for all $X\in\mathbb{R}^{n+1}_+\cap \partial B_1$
\begin{eqnarray}
\label{u1lambda}
 \partial_r \left(u_e^\lambda(X)\right)
&=& \lambda \frac{du_e^\lambda(X)}{d\lambda} -\frac{2s}{p-1} u_e^\lambda(X)
\\\label{u2lambda}
 \partial_r \left(\frac{du_e^\lambda(X)}{d\lambda}\right)
&= &\lambda \frac{d^2u_e^\lambda(X)}{d\lambda^2} +\frac{p-1-2s}{p-1} \frac{du_e^\lambda(X)}{d\lambda}
\\\label{v2lambda}
\partial_r \left( v_e^\lambda(X)\right) &=& \lambda \frac{dv_e^\lambda(X)}{d\lambda}- \frac{2(p+s-1)}{p-1} v_e^\lambda(X)
\end{eqnarray}
Substituting (\ref{u2lambda}) and (\ref{v2lambda}) in (\ref{energy2}) we get
 \begin{eqnarray}\label{energyder}
\frac{d\bar E(u_e^\lambda,1)}{d\lambda}&=&
\int_{  \mathbb{R}^{n+1}_{+}\cap \partial B_1} y^b v_e^\lambda  \left (   \lambda \frac{d^2u_e^\lambda}{d\lambda^2} +\frac{p-1-2s}{p-1} \frac{du_e^\lambda}{d\lambda}\right) -   y^b \left( \lambda \frac{dv_e^\lambda}{d\lambda}- \frac{2(p+s-1)}{p-1} v_e^\lambda \right) \frac{du_e^\lambda}{d\lambda}
\\&=& \nonumber \int_{  \mathbb{R}^{n+1}_{+}\cap \partial B_1} y^b \left( \lambda  v_e^\lambda    \frac{d^2u_e^\lambda}{d\lambda^2} +3  v_e^\lambda \frac{du_e^\lambda}{d\lambda}    -    \lambda \frac{dv_e^\lambda}{d\lambda} \frac{du_e^\lambda}{d\lambda} \right)
\end{eqnarray}
Taking derivative of (\ref{ulambda}) in $r$ we get
$$ r \frac{\partial^2 u_e^\lambda}{\partial r^2}+ \frac{\partial u_e^\lambda}{\partial r}= \lambda \frac{\partial}{\partial r}\left(\frac{du_e^\lambda}{d\lambda} \right) - \frac{2s}{p-1}  \frac{\partial u_e^\lambda}{\partial r}$$
So, from (\ref{u2lambda}) for all $X\in\mathbb{R}^{n+1}_+\cap \partial B_1$ we have
 \begin{eqnarray}\label{2ru}
\frac{\partial^2 u_e^\lambda}{\partial r^2} &=& \lambda \frac{\partial}{\partial r}\left(\frac{du_e^\lambda}{d\lambda} \right) -  \frac{p+2s-1}{p-1}  \frac{\partial u_e^\lambda}{\partial r} \\&=& \nonumber \lambda \left(  \lambda \frac{d^2u_e^\lambda}{d\lambda^2} +\frac{p-2s-1}{p-1} \frac{du_e^\lambda}{d\lambda}  \right) -   \frac{p+2s-1}{p-1}  \left(  \lambda \frac{du_e^\lambda}{d\lambda} -\frac{2s}{p-1} u_e^\lambda \right)
\\&=&\nonumber \lambda^2 \frac{d^2u_e^\lambda}{d\lambda^2} - \frac{4s}{p-1}   \lambda \frac{du_e^\lambda}{d\lambda} +\frac{2s(p+2s-1)}{(p-1)^2} u_e^\lambda
\end{eqnarray}
Note that
 \begin{eqnarray*}
v_e^\lambda= \Delta_b u^\lambda_e = y^{-b} \div(y^b \nabla u^\lambda_e)
\end{eqnarray*}
and on $\mathbb{R}^{n+1}_+\cap \partial B_1$, we have $$ \div(y^b \nabla u_e^\lambda )=(u_{rr} +(n+b)u_r )\theta_1^b +\div_{\mathcal S^n} (\theta_1^b \nabla_{S^n} u_e^\lambda)$$
where $\theta_1=\frac {y}{r}$. From the above, (\ref{u1lambda}) and (\ref{2ru}) we get
\begin{eqnarray*}
v_e^\lambda &=& \lambda^2 \frac{d^2u_e^\lambda}{d\lambda^2} +    \lambda \frac{du_e^\lambda}{d\lambda} (n+b-\frac{4s}{p-1} )+ u_e^\lambda  (\frac{2s}{p-1})(\frac{p+2s-1}{p-1}-n-b) + \theta_1^{-b}\div_{\mathcal S^n} (\theta_1^b \nabla_{S^n} u_e^\lambda)
\end{eqnarray*}
From this and (\ref{energyder}) we get
 \begin{eqnarray}
\frac{d\bar E(u_e^\lambda,1)}{d\lambda}&=& \int_{  \mathbb{R}^{n+1}_{+}\cap \partial B_1} \theta_1^b \lambda  \left(    \lambda^2 \frac{d^2u_e^\lambda}{d\lambda^2} +\alpha \lambda \frac{du_e^\lambda}{d\lambda} + \beta u_e^\lambda     \right)   \frac{d^2u_e^\lambda}{d\lambda^2}
\\&&+  \int_{  \mathbb{R}^{n+1}_{+}\cap \partial B_1}   \theta_1^b 3\left(    \lambda^2 \frac{d^2u_e^\lambda}{d\lambda^2} +\alpha \lambda \frac{du_e^\lambda}{d\lambda} + \beta u_e^\lambda     \right)          \frac{du_e^\lambda}{d\lambda}
\\&&- \int_{  \mathbb{R}^{n+1}_{+}\cap \partial B_1}     \theta_1^b  \lambda  \frac{du_e^\lambda}{d\lambda}  \frac{d}{d\lambda}  \left(    \lambda^2 \frac{d^2u_e^\lambda}{d\lambda^2} +\alpha \lambda \frac{du_e^\lambda}{d\lambda} + \beta u_e^\lambda     \right)
\\&& +\int_{  \mathbb{R}^{n+1}_{+}\cap \partial B_1} \theta_1^b \lambda   \frac{d^2u_e^\lambda}{d\lambda^2} \theta_1^{-b}\div_{\mathcal S^n} (\theta_1^b \nabla_{S^n} u_e^\lambda)
 \\&& +\int_{  \mathbb{R}^{n+1}_{+}\cap \partial B_1}  3  \theta_1^b \ \frac{du_e^\lambda}{d\lambda}  \theta_1^{-b}\div_{\mathcal S^n} (\theta_1^b \nabla_{S^n} u_e^\lambda)  \\&& -  \int_{  \mathbb{R}^{n+1}_{+}\cap \partial B_1} \theta_1^b  \lambda \frac{d}{d\lambda}\left(  \theta_1^{-b}\div_{\mathcal S^n} (\theta_1^b \nabla_{S^n} u_e^\lambda)  \right) \frac{du_e^\lambda}{d\lambda}
\end{eqnarray}
where $\alpha:=n + b- \frac{4s}{p-1}$ and $\beta:=\frac{2s}{p-1}\left(  \frac{p+2s-1}{p-1} -n-b \right)$. Simplifying the integrals we get
\begin{eqnarray}
\label{de}\frac{d\bar E(u_e^\lambda,1)}{d\lambda}&=&  \int_{  \mathbb{R}^{n+1}_{+}\cap \partial B_1}  \theta_1^b \left( 2 \lambda^3    \left(   \frac{d^2u_e^\lambda}{d\lambda^2}\right)^2 + 4 \lambda^2  \frac{d^2u_e^\lambda}{d\lambda^2}  \frac{du_e^\lambda}{d\lambda} +2(\alpha-\beta) \lambda \left(    \frac{du_e^\lambda}{d\lambda}\right)^2 \right)
 \\&& \nonumber + \int_{  \mathbb{R}^{n+1}_{+}\cap \partial B_1}  \theta_1^b \left( \frac{\beta}{2} \frac{d^2}{d\lambda^2} \left(   \lambda (u_e^\lambda)^2    \right) -\frac{1}{2} \frac{d}{d\lambda}  \left(   \lambda^3      \frac{d}{d\lambda} \left(   \frac{d  u_e^\lambda }{d\lambda} \right)^2     \right) +\frac{\beta}{2} \frac{d}{d\lambda}(u_e^\lambda)^2 \right)
\\&& \nonumber+\int_{  \mathbb{R}^{n+1}_{+}\cap \partial B_1}   \lambda \frac{d^2u_e^\lambda}{d\lambda^2} \div_{\mathcal S^n} (\theta_1^b \nabla_{S^n} u_e^\lambda)  +3  \div_{\mathcal S^n} (\theta_1^b \nabla_{S^n} u_e^\lambda)    \frac{du_e^\lambda}{d\lambda}    -    \lambda \frac{d}{d\lambda}\left(  \div_{\mathcal S^n} (\theta_1^b \nabla_{S^n} u_e^\lambda)   \right) \frac{du_e^\lambda}{d\lambda}
\end{eqnarray}
Note that from the assumptions we have  $\alpha-\beta-1>0$, therefore the first term in the RHS of (\ref{de}) is positive that is
\begin{eqnarray*}
2 \lambda^3    \left(   \frac{d^2u_e^\lambda}{d\lambda^2}\right)^2 + 4 \lambda^2  \frac{d^2u_e^\lambda}{d\lambda^2}  \frac{du_e^\lambda}{d\lambda} +2(\alpha-\beta) \lambda \left(    \frac{du_e^\lambda}{d\lambda}\right)^2
=2 \lambda \left(  \lambda  \frac{d^2u_e^\lambda}{d\lambda^2}  +   \frac{du_e^\lambda}{d\lambda}  \right)^2 +2(\alpha-\beta-1) \lambda \left(    \frac{du_e^\lambda}{d\lambda}\right)^2 >0
\end{eqnarray*}
From this we have
\begin{eqnarray*}
\label{de1}\frac{d\bar E(u_e^\lambda,1)}{d\lambda}& \ge &  \int_{  \mathbb{R}^{n+1}_{+}\cap \partial B_1}  \theta_1^b\left( \frac{\beta}{2} \frac{d^2}{d\lambda^2} \left(   \lambda (u_e^\lambda)^2    \right) -\frac{1}{2} \frac{d}{d\lambda}  \left(   \lambda^3      \frac{d}{d\lambda} \left(   \frac{d  u_e^\lambda }{d\lambda} \right)^2     \right) +\frac{\beta}{2} \frac{d}{d\lambda}(u_e^\lambda)^2 \right)
\\&& \nonumber+\int_{  \mathbb{R}^{n+1}_{+}\cap \partial B_1}   \lambda \frac{d^2u_e^\lambda}{d\lambda^2} \div_{\mathcal S^n} (\theta_1^b \nabla_{S^n} u_e^\lambda)  +3  \div_{\mathcal S^n} (\theta_1^b \nabla_{S^n} u_e^\lambda)    \frac{du_e^\lambda}{d\lambda}    -    \lambda \frac{d}{d\lambda}\left(  \div_{\mathcal S^n} (\theta_1^b \nabla_{S^n} u_e^\lambda)   \right) \frac{du_e^\lambda}{d\lambda} \\&=:& R_1+R_2.
\end{eqnarray*}
Note that the terms appeared in $R_1$ are of the following form
\begin{eqnarray*}
 \int_{  \mathbb{R}^{n+1}_{+}\cap \partial B_1}
\theta_1^b \frac{d^2}{d\lambda^2} \left(   \lambda (u_e^\lambda)^2    \right) &=&  \frac{d^2}{d\lambda^2} \left(   \lambda^{ \frac{4s}{p-1}+2(s-1)-n } \int_{  \mathbb{R}^{n+1}_{+}\cap \partial B_\lambda}
y^b  u_e^2    \right) \\
 \int_{  \mathbb{R}^{n+1}_{+}\cap \partial B_1}
\theta_1^b \frac{d}{d\lambda}  \left[  \lambda^3      \frac{d}{d\lambda} \left(   \frac{d  u_e^\lambda }{d\lambda} \right)^2     \right] &=&  \frac{d}{d\lambda}  \left[   \lambda^3      \frac{d}{d\lambda} \left(  \lambda^{ \frac{4s}{p-1}+2s-3-n }    \int_{  \mathbb{R}^{n+1}_{+}\cap \partial B_\lambda} y^b \left[  \frac{2s}{p-1} \lambda^{-1} u_e + \frac{\partial u_e}{\partial r}      \right]^2     \right)   \right]
 \\
  \int_{  \mathbb{R}^{n+1}_{+}\cap \partial B_1}  y^b \frac{d}{d\lambda}(u_e^\lambda)^2 &=& \frac{d}{d\lambda} \left( \lambda^{  2s-3+ \frac{4s}{p-1}-n }  \int_{  \mathbb{R}^{n+1}_{+}\cap \partial B_\lambda}      y^b u_e ^2\right)
\end{eqnarray*}
We now apply integration by parts to simplify the terms appeared in $R_2$.
\begin{eqnarray*}
 R_2 &=&  \int_{  \mathbb{R}^{n+1}_{+}\cap \partial B_1}   \lambda \frac{d^2u_e^\lambda}{d\lambda^2} \div_{\mathcal S^n} (\theta_1^b \nabla_{S^n} u_e^\lambda)  +3  \div_{\mathcal S^n} (\theta_1^b \nabla_{S^n} u_e^\lambda)    \frac{du_e^\lambda}{d\lambda}    -    \lambda \frac{d}{d\lambda}\left(  \div_{\mathcal S^n} (\theta_1^b \nabla_{S^n} u_e^\lambda)   \right) \frac{du_e^\lambda}{d\lambda} \\
 &=&  \int_{  \mathbb{R}^{n+1}_{+}\cap \partial B_1} - \theta_1^b  \lambda  \nabla_{\mathcal S^n} u_e^\lambda \cdot   \nabla_{\mathcal S^n}  \frac{d^2u_e^\lambda}{d\lambda^2} - 3 \theta_1^b   \nabla_{\mathcal S^n}  u_e^\lambda \cdot    \nabla_{\mathcal S^n}  \frac{du_e^\lambda}{d\lambda}  +  \theta_1^b\lambda \left|     \nabla_{\mathcal S^n}  \frac{du_e^\lambda}{d\lambda}   \right|^2\\
&=& - \frac{\lambda}{2}  \frac{d^2}{d\lambda^2} \left(  \int_{  \mathbb{R}^{n+1}_{+}\cap \partial B_1} \theta_1^b  |\nabla_\theta u_e^\lambda |^2   \right) -\frac{3}{2} \frac{d}{d\lambda} \left(  \int_{  \mathbb{R}^{n+1}_{+}\cap \partial B_1}  \theta_1^b |\nabla_\theta u_e^\lambda |^2   \right)+2\lambda \int_{  \mathbb{R}^{n+1}_{+}\cap \partial B_1}  \theta_1^b \left|    \nabla_\theta \frac{du_e^\lambda}{d\lambda}   \right|^2
\\&=& - \frac{1}{2}  \frac{d^2}{d\lambda^2} \left(   \lambda      \int_{  \mathbb{R}^{n+1}_{+}\cap \partial B_1} \theta_1^b  |\nabla_\theta u_e^\lambda |^2   \right) -\frac{1}{2} \frac{d}{d\lambda} \left(  \int_{  \mathbb{R}^{n+1}_{+}\cap \partial B_1} \theta_1^b  |\nabla_\theta u_e^\lambda |^2   \right)+2\lambda \int_{  \mathbb{R}^{n+1}_{+}\cap \partial B_1} \theta_1^b \left|    \nabla_\theta \frac{du_e^\lambda}{d\lambda}   \right|^2
\\&\ge& - \frac{1}{2}  \frac{d^2}{d\lambda^2} \left(   \lambda      \int_{  \mathbb{R}^{n+1}_{+}\cap \partial B_1}   \theta_1^b |\nabla_\theta u_e^\lambda |^2   \right) -\frac{1}{2} \frac{d}{d\lambda} \left(  \int_{  \mathbb{R}^{n+1}_{+}\cap \partial B_1}  \theta_1^b |\nabla_\theta u_e^\lambda |^2   \right)
 \end{eqnarray*}
 Note that the two terms that appear as lower bound for $R_3$ are of the form
  \begin{eqnarray*}
 \frac{d^2}{d\lambda^2} \left(   \lambda      \int_{  \mathbb{R}^{n+1}_{+}\cap \partial B_1} \theta_1^b  |\nabla_\theta u_e^\lambda |^2   \right) &=&  \frac{d^2}{d\lambda^2} \left[  \lambda^{2s\frac{p+1}{p-1}-n}   \int_{  \mathbb{R}^{n+1}_{+}\cap \partial B_\lambda} y^b  \left(  |\nabla  u|^2-\left|     \frac{\partial u}{\partial r}\right|^2  \right)
  \right]
  \\
   \frac{d}{d\lambda} \left(  \int_{  \mathbb{R}^{n+1}_{+}\cap \partial B_1}  \theta_1^b |\nabla_\theta u_e^\lambda |^2   \right) &=&  \frac{d}{d\lambda} \left[  \lambda^{2s\frac{p+1}{p-1}-n-1}     \int_{  \mathbb{R}^{n+1}_{+}\cap \partial B_\lambda} y^b \left(  |\nabla  u|^2-\left|     \frac{\partial u}{\partial r}\right|^2  \right)      \right]
  \end{eqnarray*}

 \hfill $ \Box$

 \begin{remark}
 It is straightforward to show that $n>\frac{p+1}{p-1} 2s$ implies $n>\frac{p+4s-1}{p+2s-1}+ \frac{2s}{p-1}-b$.
 \end{remark}

 \section{Homogeneous Solutions}
 In this section, we examine homogenous solutions of the form $u=r^{-\frac{2s}{p-1}} \psi(\theta)$. Note that the methods and ideas that we apply here are  different from the ones used in \cite{ddw}.
 \begin{thm}\label{homog}
 Suppose that $u=r^{-\frac{2s}{p-1}} \psi(\theta)$ is a stable solution of (\ref{main}) then $\psi=0$ provided
 $p>\frac{n+2s}{n-2s}$ and $$ p \frac{\Gamma(\frac{n}{2}-\frac{s}{p-1}) \Gamma(s+\frac{s}{p-1})}{\Gamma(\frac{s}{p-1}) \Gamma(\frac{n-2s}{2}-\frac{s}{p-1})}>\frac{ \Gamma(\frac{n+2s}{4})^2 }{\Gamma(\frac{n-2s}{4})^2}$$
 \end{thm}
 \begin{proof}
 Since $u$ satisfies (\ref{main}), the function $\psi$ satisfies (we omit the P.V.)
 \begin{eqnarray*}
 |x|^{-\frac{2ps}{p-1}} \psi^p(\theta)
&=& \int \frac{  |x|^{-\frac{2s}{p-1}} \psi(\theta) -|y|^{-\frac{2s}{p-1}} \psi(\sigma) }{ |x-y|^{n+2s}} dy \\&=&
  \int \frac{  |x|^{-\frac{2s}{p-1}} \psi(\theta) -r^{-\frac{2s}{p-1}} t^{-\frac{2s}{p-1}} \psi(\sigma) }{ (t^2+1- 2t <\theta,\sigma>)^{\frac{n+2s}{2}} |x|^{n+2s}} |x|^n t^{n-1} dt d\sigma \text{ \ where \ } |y|=rt
\\ &=&  |x|^{-\frac{2ps}{p-1}}  [  \int \frac{  \psi(\theta) - t^{-\frac{2s}{p-1}} \psi(\theta) }{ (t^2+1- 2t <\theta,\sigma>)^{\frac{n+2s}{2}} }  t^{n-1} dt d\sigma \\&&+ \int \frac{  t^{-\frac{2s}{p-1}} (\psi(\theta) -  \psi(\sigma) }{ (t^2+1- 2t <\theta,\sigma>)^{\frac{n+2s}{2}} }  t^{n-1} dt d\sigma]
\end{eqnarray*}
 We now drop $|x|^{-\frac{2ps}{p-1}}$ and get
 \begin{equation}\label{Ans}
 \psi(\theta) A_{n,s}(\theta) + \int_{\mathbb S^{n-1}} K_{\frac{2s}{p-1}} (<\theta,\sigma>) (\psi(\theta)-\psi(\sigma)) d \sigma= \psi^p(\theta)
 \end{equation}
 where $$A_{n,s}:=\int_0^\infty \int_{\mathbb S^{n-1}} \frac{1-t^{  -\frac{2s}{p-1} }}{    (t^2+1- 2t <\theta,\sigma>)^{\frac{n+2s}{2}} } t^{n-1} d\sigma dt$$ and $$ K_{\frac{2s}{p-1}}(<\theta,\sigma>):=\int_0^\infty \frac{t^{  n-1-\frac{2s}{p-1} }}{    (t^2+1- 2t <\theta,\sigma>)^{\frac{n+2s}{2}} }  dt  $$
 Note that
  \begin{eqnarray*}
K_{\frac{2s}{p-1}}(<\theta,\sigma>)&=& \int_0^1 \frac{t^{  n-1-\frac{2s}{p-1} }}{    (t^2+1- 2t <\theta,\sigma>)^{\frac{n+2s}{2}} }  dt +\int_1^\infty \frac{t^{  n-1-\frac{2s}{p-1} }}{    (t^2+1- 2t <\theta,\sigma>)^{\frac{n+2s}{2}} }  dt
\\&=& \int_0^1 \frac{t^{  n-1-\frac{2s}{p-1}} + t^{  2s-1+\frac{2s}{p-1}}}{    (t^2+1- 2t <\theta,\sigma>)^{\frac{n+2s}{2}} }  dt
 \end{eqnarray*}
We now set $K_{\alpha}(<\theta,\sigma>)= \int_0^1 \frac{t^{  n-1+\alpha} + t^{  2s-1+\alpha}}{    (t^2+1- 2t <\theta,\sigma>)^{\frac{n+2s}{2}} }  dt$. The most important property of the $K_\alpha$ is that $K_\alpha$ is decreasing in $\alpha$. This can be seen by the following elementary calculations
   \begin{eqnarray*}
\partial_\alpha K_\alpha&=& \int_0^1 \frac{-t^{  n-1-\alpha} \ln t + t^{  2s-1+\alpha} \ln t}{    (t^2+1- 2t <\theta,\sigma>)^{\frac{n+2s}{2}} }  dt
\\&=& \int_0^1 \frac{\ln t (-t^{  n-1-\alpha} + t^{  2s-1+\alpha})}{    (t^2+1- 2t <\theta,\sigma>)^{\frac{n+2s}{2}} }  dt<0
 \end{eqnarray*}
For the last part we have used the fact that for $p>\frac{n+2s}{n-2s}$ we have $2s-1+\alpha<n-1-\alpha$.

From (\ref{Ans}) we get the following
 \begin{equation}\label{Ans2}
\int_{\mathbb S^{n-1}} \psi^2(\theta) A_{n,s} + \int_{\mathbb S^{n-1}} K_{\frac{2s}{p-1}} (<\theta,\sigma>) (\psi(\theta)-\psi(\sigma))^2 d\theta d \sigma= \int_{\mathbb S^{n-1}} \psi^{p+1}(\theta) d\theta
 \end{equation}
 We set a standard cut-off function $\eta_\epsilon\in C_c^1(\mathbb R_+)$ at the origin and at infinity that is $\eta_\epsilon=1$ for $\epsilon<r<\epsilon^{-1}$ and $\eta_\epsilon=0$ for either $r<\epsilon/2$ or $r>2/\epsilon$. We test the stability (\ref{stability}) on the function $\phi(x)=r^{-\frac{n-2s}{2}} \psi(\theta) \eta_\epsilon(r)$.

 Note that
    \begin{eqnarray*}
\int_{\mathbb R^n} \frac{\phi(x)-\phi(y)}{|x-y|^{n+2s}}  dy = \int \int_{\mathbb S^{n-1}} \frac{r^{-\frac{n-2s}{2} } \psi(\theta)\eta(r) -|y|^{-\frac{n-2s}{2} } \psi(\sigma) \eta(|y|)}{    (r^2+|y|^2- 2r |y| <\theta,\sigma>)^{\frac{n+2s}{2}} }  d\sigma d(|y|)
\end{eqnarray*}
Now set $|y|=rt$ then
 \begin{eqnarray*}
\int_{\mathbb R^n} \frac{\phi(x)-\phi(y)}{|x-y|^{n+2s}}  dy &=& r^{-\frac{n}{2} -s} \int_0^\infty \int_{\mathbb S^{n-1}} \frac{\psi(\theta)\eta(r) - t^{-\frac{n-2s}{2} } \psi(\sigma) \eta(rt)}{  (t^2+1- 2t <\theta,\sigma>)^{\frac{n+2s}{2}  }} t^{n-1} dt d\sigma
\\&=& r^{-\frac{n}{2} -s} \int \int_{\mathbb S^{n-1}} \frac{\psi(\theta)\eta(r) - t^{-\frac{n-2s}{2} } \psi(\sigma) \eta(r) +t^{-\frac{n-2s}{2} } (   \eta(r) \psi(\theta)-\eta(rt) \psi(\sigma)  )  }{  (t^2+1- 2t <\theta,\sigma>)^{\frac{n+2s}{2}  }} t^{n-1} dt d\sigma
\\&=& r^{-\frac{n}{2} -s} \eta(r) \psi(\theta) \int_0^\infty \int_{\mathbb S^{n-1}} \frac{1-t^{  \frac{n-2s}{2} }}{    (t^2+1- 2t <\theta,\sigma>)^{\frac{n+2s}{2}} }   t^{n-1} dt d\sigma
\\&&+ r^{-\frac{n}{2} -s} \eta(r)  \int_0^\infty \int_{\mathbb S^{n-1}} \frac{t^{ n-1 -\frac{n-2s}{2} } (\psi(\theta)-\psi(\sigma))}{    (t^2+1- 2t <\theta,\sigma>)^{\frac{n+2s}{2}} }  dt d\sigma
\\&&+  r^{-\frac{n}{2} -s} \int_0^\infty \int_{\mathbb S^{n-1}} \frac{t^{n-1  -\frac{n-2s}{2} } (\eta(r)-\eta(rt))\psi(\sigma) }{    (t^2+1- 2t <\theta,\sigma>)^{\frac{n+2s}{2}} } dt d\sigma
\end{eqnarray*}
Define $\Lambda_{n,s} :=\int_0^\infty  \int_{\mathbb S^{n-1}}\frac{1-t^{  \frac{n-2s}{2} }}{    (t^2+1- 2t <\theta,\sigma>)^{\frac{n+2s}{2}} }  t^{n-1} d\sigma dt$. Therefore,
 \begin{eqnarray*}
\int_{\mathbb R^n} \frac{\phi(x)-\phi(y)}{|x-y|^{n+2s}}  dy &=&r^{-\frac{n}{2} -s} \eta(r) \psi(\theta) \Lambda_{n,s}
\\&&+ r^{-\frac{n}{2} -s} \eta(r)  \int_{\mathbb S^{n-1}} K_{\frac{n-2s}{2}}(<\theta,\sigma>) (\psi(\theta)-\psi(\sigma)) d\sigma
\\&&+r^{-\frac{n}{2} -s} \int_0^\infty \int_{\mathbb S^{n-1}} \frac{t^{  -\frac{n-2s}{2} } (\eta(r)-\eta(rt))\psi(\sigma) }{    (t^2+1- 2t <\theta,\sigma>)^{\frac{n+2s}{2}} } dt d\sigma
\end{eqnarray*}
 Applying the above, we compute the left-hand side of the stability inequality (\ref{stability}),
\begin{eqnarray}\label{stable1}
\nonumber\int_{\mathbb R^n}\int_{\mathbb R^n} \frac{ (\phi(x)-\phi(y))^2}{|x-y|^{n+2s}} dx dy &=& 2\int_{\mathbb R^n}\int_{\mathbb R^n} \frac{ (\phi(x)-\phi(y))\phi(x)}{|x-y|^{n+2s}} dx dy
\nonumber\\&=&
2\int_0^\infty r^{-1} \eta^2(r) dr \int_{\mathbb S^{n-1}} \psi^2  \Lambda_{n,s} d\theta
\nonumber\\&&+ 2\int_0^\infty r^{-1} \eta^2(r) dr \int_{\mathbb S^{n-1}} K_{\frac{n-2s}{2}}(<\theta,\sigma>) (\psi(\theta)-\psi(\sigma))^2 d\sigma d\theta
\nonumber\\&&+2 \int_0^\infty  \left[ \int_0^\infty   r^{-1} \eta(r) (\eta(r)-\eta(rt)) dr  \right] \int_{\mathbb S^{n-1}} \int_{\mathbb S^{n-1}} \frac{   t^{ n- 1-\frac{n-2s}{2} } \psi(\sigma)\psi(\theta)  }{ (t^2+1- 2t <\theta,\sigma>)^{\frac{n+2s}{2}}}  d\sigma d\theta dt
 \end{eqnarray}
We now compute the second term in the stability inequality (\ref{stability}) for the test function $\phi(x)=r^{-\frac{n-2s}{2}} \psi(\theta) \eta(r)$ and $u=r^{-\frac{2s}{p-1}} \psi(\theta)$,
\begin{eqnarray}\label{stable2}
\nonumber p \int_0^\infty |u|^{p-1} \phi^2 &=& p \int_0^\infty r^{-2s} r^{-(n-2s)} \psi^{p+1} \eta^2(r) dr \\&=&
p \int_0^\infty r^{-1} \eta^2(r)  dr \int_{\mathbb S^{n-1}} \psi^{p+1}  (\theta) d\theta
  \end{eqnarray}
Due to the definition of the $\eta_\epsilon$, we have  $\int_0^\infty r^{-1} \eta_\epsilon^2(r)  dr=\ln (2/\epsilon) +O(1)$. Note that this term appears in both terms of the stability inequality that we computed  in (\ref{stable1}) and (\ref{stable2}). We now claim that
$$f_\epsilon(t):=\int_0^\infty   r^{-1} \eta_\epsilon(r) (\eta_\epsilon(r)-\eta_\epsilon(rt)) dr=O(\ln t)$$
Note that $\eta_\epsilon(rt)=1$ for $\frac\epsilon t<r<\frac{1}{t\epsilon}$ and $\eta_\epsilon(rt)=0$ for either $r<\frac{\epsilon}{2t}$ or $r>\frac{2}{t\epsilon}$. Now consider various ranges of value of $t\in (0,\infty)$ to compare the support of $\eta_\epsilon(r)$ and  $\eta_\epsilon(rt)$. From the definition of $\eta_\epsilon$, we have
$$f_\epsilon(t)=\int_{\frac{\epsilon}{2}}^{ \frac{2}{\epsilon}    }   r^{-1} \eta_\epsilon(r) (\eta_\epsilon(r)-\eta_\epsilon(rt)) dr$$
In what follows we consider a few cases to explain  the claim. For example when  $\epsilon< \frac{\epsilon}{t}< \frac{1}{\epsilon}$ then
$$f_\epsilon(t)\approx \int_{\frac{\epsilon}{2}}^{ \frac{\epsilon}{t}    }   r^{-1} dr +  \int_{\frac{1}{\epsilon}}^{ \frac{2}{\epsilon t}    }   r^{-1} dr \approx \ln t $$
Now consider the case $\frac{1}{\epsilon}< \frac{\epsilon}{t}< \frac{1}{\epsilon}$ then $t \approx \epsilon^2$. So,
$$f_\epsilon(t)\approx \int_{\frac{\epsilon}{2}}^{ \frac{\epsilon}{t}    }   r^{-1} dr +  \int_{\frac{\epsilon}{t}}^{ \frac{2}{\epsilon }    }   r^{-1} dr \approx \ln t + \ln \epsilon \approx \ln t$$
Other cases can be treated similarly. From this one can see that
\begin{eqnarray}\label{stable2}
&&\int_0^\infty  \left[ \int_0^\infty   r^{-1} \eta(r) (\eta(r)-\eta(rt)) dr  \right] \int_{\mathbb S^{n-1}}\int_{\mathbb S^{n-1}} \frac{   t^{ n- 1-\frac{n-2s}{2} }   }{ (t^2+1- 2t <\theta,\sigma>)^{\frac{n+2s}{2}}} \psi(\sigma)\psi(\theta) d\sigma d\theta dt
\\&\approx&  \int_{\mathbb S^{n-1}} \int_{\mathbb S^{n-1}} \int_0^\infty \frac{   t^{ n- 1-\frac{n-2s}{2} }  \ln t }{ (t^2+1- 2t <\theta,\sigma>)^{\frac{n+2s}{2}}} \psi(\sigma)\psi(\theta)dt  d\sigma d\theta
\\&=& O(1)
  \end{eqnarray}
 Collecting higher order terms  of the stability inequality we get
 \begin{equation}
  \Lambda_{n,s} \int_{\mathbb S^{n-1}} \psi^2 + \int_{\mathbb S^{n-1}} K_{\frac{n-2s}{2}}(<\theta,\sigma>) (\psi(\theta)-\psi(\sigma))^2 d\sigma \ge p  \int_{\mathbb S^{n-1}} \psi^{p+1}
 \end{equation}
 From this and (\ref{Ans2}) we obtain
  \begin{eqnarray*}
  ( \Lambda_{n,s} -p A_{n,s}) \int_{\mathbb S^{n-1}} \psi^2 + \int_{\mathbb S^{n-1}} (K_{\frac{n-2s}{2}} - pK_{\frac{2s}{p-1}}  )(<\theta,\sigma>) (\psi(\theta)-\psi(\sigma))^2 d\sigma \ge 0
  \end{eqnarray*}
Note that $K_\alpha$ is decreasing in $\alpha$. This implies $K_{\frac{n-2s}{2}} < K_{\frac{2s}{p-1}}$ for $p>\frac{n+2s}{n-2s}$. So, $K_{\frac{n-2s}{2}} - pK_{\frac{2s}{p-1}} <0$. On the other hand the assumption of the theorem implies that $\Lambda_{n,s} -p A_{n,s}<0$. Therefore, $\psi=0$.

 \end{proof}

 \begin{remark} Note that in this section we never used the fact that $1<s<2$. So this proof holds for a larger range of the parameter $s$.
\end{remark}

 \section{Energy Estimates}
 In this section, we provide some estimates for solutions of (\ref{main}). These estimates are needed in the next section when we perform a blow-down analysis argument.  The methods and ideas provided in this section are strongly motivated by \cite{ddww,ddw}.

 \begin{lemma}\label{iden} The following identities hold for any functions $\zeta$ and $\eta$,
 \begin{eqnarray}
 \Delta_b \zeta \Delta_b(\zeta \eta^2)- |\Delta_b(\zeta\eta)|^2 &=& -\zeta^2 |\Delta_b \eta|^2 +2\zeta \Delta_b \zeta |\nabla \eta|^2 -4 |\nabla \zeta\cdot\nabla \eta|^2 -4 \zeta\Delta_b\eta \nabla \zeta\cdot\nabla \eta
 \\ \Delta_b (\zeta \eta) &=& \eta \Delta_b \zeta  +\zeta \Delta_b \eta +2\nabla\zeta\cdot\nabla\eta
 \end{eqnarray}
 \end{lemma}
 \begin{proof}
 We omit the proof, since it is elementary.
\end{proof}
 We apply the given identities to get some energy estimates.
 \begin{lemma}\label{estimate}  Let $u$ be a solution of (\ref{main}) that is stable outside a ball $B_{R_0}$ and $u_e$ satisfies (\ref{maine}). Then there exists a positive constant $C$  such that
  \begin{eqnarray}\label{estimateeq}
\int_{\partial \mathbb R_+^{n+1}}  |u_e|^{p+1} \eta^2 + \int_{ \mathbb R_+^{n+1}}  y^b |\Delta_b u_e|^2 \eta^2 &\le& C \int_{\mathbb R_+^{n+1}} y^b u_e^2 \left( |\Delta_b \eta|^2 +|\Delta_b |\nabla \eta|^2|+|\nabla\eta\cdot\nabla\Delta_b\eta|\right)  \\&&+ C \int_{\mathbb R_+^{n+1}} y^b  |u_e| |\Delta_b u_e| |\nabla \eta|^2
 \end{eqnarray}
 \end{lemma}
 \begin{proof} Multiply the equation with $y^{b} u\eta^2 $ where $\eta$ is a test function to get
  \begin{eqnarray*}
0&=&\int_{ \mathbb R_+^{n+1}}  y^b (u_e\eta^2) \Delta^2_b u_e = \int_{ \mathbb R_+^{n+1}}   u_e\eta^2 \div(y^b \nabla\Delta_b u_e) \\&=& - \int_{ \mathbb R_+^{n+1}} y^b  \nabla (u_e\eta^2)\cdot  \nabla\Delta_b u_e + \int_{ \partial\mathbb R_+^{n+1}} \lim_{y\to 0} y^b \partial_y (\Delta_b u_e) (u_e\eta^2)
\\&=& - \int_{ \mathbb R_+^{n+1}} y^b  \nabla (u_e\eta^2)\cdot  \nabla\Delta_b u_e + C_{n,s} \int_{ \partial\mathbb R_+^{n+1}} |u_e|^{p+1} \eta^2
 \end{eqnarray*}
From this we get
  \begin{equation}\label{delta}
C_{n,s} \int_{ \partial\mathbb R_+^{n+1}} |u_e|^{p+1} \eta^2= \int_{ \mathbb R_+^{n+1}} y^b \Delta_b u_e \Delta_b(u_e\eta^2)
 \end{equation}
Apply Lemma \ref{iden} for $\zeta=u_e$ we get
  \begin{eqnarray}\label{up+1}
\nonumber C_{n,s} \int_{ \partial\mathbb R_+^{n+1}} |u_e|^{p+1} \eta^2 &=& \int_{ \mathbb R_+^{n+1}} y^b |\Delta_b(u_e\eta)|^2  - \int_{ \mathbb R_+^{n+1}} y^b u_e^2 |\Delta_b \eta|^2 +2\int_{ \mathbb R_+^{n+1}} y^b u_e \Delta_b u_e |\nabla \eta|^2 \\&& -4\int_{ \mathbb R_+^{n+1}} y^b  |\nabla u_e\cdot\nabla \eta|^2 -4 \int_{ \mathbb R_+^{n+1}} y^b u_e\Delta_b\eta \nabla u_e\cdot\nabla \eta
 \end{eqnarray}
Note that the last integral is
  \begin{eqnarray*}
-4 \int_{ \mathbb R_+^{n+1}} y^b u_e\Delta_b\eta \nabla u_e\cdot\nabla \eta &=& -2 \int_{ \mathbb R_+^{n+1}} y^b \Delta_b\eta \nabla (u_e^2) \cdot\nabla \eta \\ &=&2 \int_{ \mathbb R_+^{n+1}} u_e^2 \div(y^b \Delta_b\eta\nabla \eta)= 2 \int_{ \mathbb R_+^{n+1}} y^b u_e^2 (|\Delta_b \eta|^2 + \nabla\eta\cdot\nabla\Delta_b \eta)
 \end{eqnarray*}
From this and (\ref{up+1}) we get
\begin{eqnarray}\label{up+1-}
C_{n,s} \int_{ \partial\mathbb R_+^{n+1}} |u_e|^{p+1} \eta^2 &=&  \int_{ \mathbb R_+^{n+1}} y^b |\Delta_b(u_e\eta)|^2   +2\int_{ \mathbb R_+^{n+1}} y^b u_e \Delta_b u_e |\nabla \eta|^2 \\&& -4\int_{ \mathbb R_+^{n+1}} y^b  |\nabla u_e\cdot\nabla \eta|^2 + \int_{ \mathbb R_+^{n+1}} y^b u_e^2 (|\Delta_b \eta|^2 + 2\nabla\eta\cdot\nabla\Delta_b \eta )
 \end{eqnarray}
We now apply the stability inequality  (\ref{stability}) for $\phi=u\eta$ to get
\begin{equation}\label{stability1}
p \int_{\mathbb R^n} |u|^{p+1} \eta^2 \le  \int_{ \mathbb R_+^{n+1}} y^b |\Delta_b(u_e\eta)|^2
\end{equation}
From (\ref{stability1}) and (\ref{up+1-}) we obtain
\begin{eqnarray}\label{up+1=}
\nonumber \int_{ \partial\mathbb R_+^{n+1}} |u_e|^{p+1} \eta^2 + \int_{ \mathbb R_+^{n+1}} y^b |\Delta_b(u_e\eta)|^2 &\le & C  \int_{ \mathbb R_+^{n+1}} y^b |u_e| |\Delta_b u_e| |\nabla \eta|^2 + C \int_{ \mathbb R_+^{n+1}} y^b  |\nabla u_e|^2 |\nabla \eta|^2  \\&&+C \int_{ \mathbb R_+^{n+1}} y^b u_e^2 (|\Delta_b \eta|^2 + |\nabla\eta\cdot\nabla\Delta_b \eta| )
 \end{eqnarray}
Note that from Lemma \ref{iden} we have $\Delta_b (u_e \eta) = \eta \Delta_b u_e  +u_e \Delta_b \eta +2\nabla u_e\cdot\nabla\eta$. So from (\ref{up+1=}) we get
\begin{eqnarray}\label{up+1+}
 \int_{ \partial\mathbb R_+^{n+1}} |u_e|^{p+1} \eta^2 + \int_{ \mathbb R_+^{n+1}} y^b |\Delta_b u_e|^2 \eta^2
&\le & C  \int_{ \mathbb R_+^{n+1}} y^b |u_e| |\Delta_b u_e| |\nabla \eta|^2 + C \int_{ \mathbb R_+^{n+1}} y^b  |\nabla u_e|^2 |\nabla \eta|^2  \\&&+C \int_{ \mathbb R_+^{n+1}} y^b u_e^2 (|\Delta_b \eta|^2 + |\nabla\eta\cdot\nabla\Delta_b \eta| )
 \end{eqnarray}
Note also that $2  |\nabla u_e|^2= \Delta_b (u_e^2) - 2u_e \Delta_b u_e$. Therefore,
\begin{eqnarray}\label{}
2  \int_{ \mathbb R_+^{n+1}} y^b  |\nabla u_e|^2 |\nabla \eta|^2 &=&  \int_{ \mathbb R_+^{n+1}} y^b  |\nabla \eta|^2 \Delta_b (u_e^2) - 2 \int_{ \mathbb R_+^{n+1}}y^b  u_e \Delta_b u_e|\nabla \eta|^2
\\&=&  \int_{ \mathbb R_+^{n+1}} y^b  u_e^2  \Delta_b |\nabla \eta|^2  - 2 \int_{ \mathbb R_+^{n+1}} y^b u_e \Delta_b u_e|\nabla \eta|^2
 \end{eqnarray}
From this and (\ref{up+1+}) we get
\begin{eqnarray}\label{}
\nonumber \int_{ \partial\mathbb R_+^{n+1}} |u_e|^{p+1} \eta^2 + \int_{ \mathbb R_+^{n+1}} y^b |\Delta_b u_e|^2 \eta^2 &\le & C  \int_{ \mathbb R_+^{n+1}} y^b |u_e| |\Delta_b u_e| |\nabla \eta|^2 \\&&+C \int_{ \mathbb R_+^{n+1}} y^b u_e^2 (|\Delta_b \eta|^2 + |\nabla\eta\cdot\nabla\Delta_b \eta| +|\Delta_b |\nabla\eta|^2| )
 \end{eqnarray}
This finishes the proof.
\end{proof}
 \begin{cor}\label{ue2} With the same assumption as Lemma \ref{estimate}. Then there exists a positive constant $C$  such that
  \begin{equation}
\int_{B_R\cap\partial\mathbb R_+^{n+1}}  |u_e|^{p+1} + \int_{B_R\cap\mathbb R_+^{n+1}}  y^b |\Delta_b u_e|^2  \le C R^{-4} \int_{B_R\cap\mathbb R_+^{n+1}} y^b u_e^2
  \end{equation}
\end{cor}
 \begin{proof}
This is a direct consequence of the  estimate (\ref{estimateeq}).  Substitute $\eta$ with $\eta^m$ in  (\ref{estimateeq}) for a number $3<m\in\mathbb N$. Therefore
 \begin{eqnarray}\label{}
 m^2 \int_{ \mathbb R_+^{n+1}} y^b |u_e| |\Delta_b u_e| |\nabla \eta|^2 \eta^{2m-2} \le \epsilon \int_{ \mathbb R_+^{n+1}} y^b |\Delta_b u_e|^2  w\eta^{2m} + C(\epsilon) \int_{ \mathbb R_+^{n+1}} y^b u_e^2 \eta^{2m-4} |\nabla \eta|^4
  \end{eqnarray}
  for a small enough $\epsilon>0$. One can apply the standard test function to finish the proof.
 \end{proof}
 \begin{lemma}\label{u2}
 Suppose that $u$ is a solution of  (\ref{main}) that is stable outside some ball $B_{R_0}\subset \mathbb R^n$. For $\eta\in C_c^\infty(\mathbb R^n\setminus \overline{B_{R_0}})$ and $x\in \mathbb R^n$ define
 \begin{equation}
 \rho(x)=\int_{\mathbb R^n} \frac{(\eta(x)-\eta(y))^2}{|x-y|^{n+2s}} dy.
 \end{equation}
 Then
 \begin{equation}
 \int_{\mathbb R^n} |u|^{p+1} \eta^2 dx +  \int_{\mathbb R^n }\int_{\mathbb R^n } \frac{| u(x)\eta(x)-u(y)\eta(y)|^2}{|x-y|^{n+2s}} dx dy \le C \int_{\mathbb R^n} u^2 \rho dx
 \end{equation}
 \end{lemma}
 \begin{proof}
 Proof is quite similar to Lemma 2.1 in \cite{ddw} and we omit it here.
  \end{proof}
 \begin{lemma}\label{rho}
 Let $m>n/2$ and $x\in\mathbb R^n$. Set
  \begin{equation}\label{eta}
 \rho(x)=\int_{\mathbb R^n} \frac{(\eta(x)-\eta(y))^2}{|x-y|^{n+2s}} dy \ \ \text{where} \ \ \eta(x)=(1+|x|^2)^{-m/2}
 \end{equation}
Then there is a constant $C=C(n,s,m)>0$ such that
\begin{equation}
 C^{-1} (1+|x|^2)^{-n/2-s}\le \rho(x)\le C (1+|x|^2)^{-n/2-s}
 \end{equation}
  \end{lemma}
 \begin{proof}
 Proof is quite similar to Lemma 2.2 in \cite{ddw} and we omit it here.
  \end{proof}

\begin{cor}\label{rhoR}
Suppose that $m>n/2$, $\eta$ given by (\ref{eta}) and $R>R_0>1$. Define
\begin{equation}\label{etaR}
 \rho_R(x)=\int_{\mathbb R^n} \frac{(\eta_R(x)-\eta_R(y))^2}{|x-y|^{n+2s}} dy \ \
 \text{where} \ \ \eta_R(x)=\eta(x/R) \psi(x/{R_0})
 \end{equation}
for the standard test function $\psi $ that is $\psi\in C^\infty(\mathbb R^n)$ and $0\le \psi\le 1$, $\psi=0$ on $B_1$ and $\psi=1$ on $\mathbb R^n\setminus B_2$. Then there exists a constant $C>0$ such that
$$\rho_R(x)\le C \eta^2(x/R) |x|^{-(n+2s)}+R^{-2s}\rho(x/R).$$
\end{cor}

\begin{lemma}\label{finalu2}
Suppose that $u$ is a solution of (\ref{main}) that is stable outside a ball $B_{R_{0}}$. Consider $\rho_R$ that is defined in Corollary \ref{} for $n/2<m<n/2+s(p+1)/2$. Then there exists a constant $C>0$ such that
$$ \int_{\mathbb R^n} u^2 \rho_R \le C \left(\int_{B_{3R_0}} u^2 \rho_R + R^{n-2s \frac{p+1}{p-1}}\right)$$
for any $R>3R_0$
\end{lemma}
 \begin{proof}
 Proof is quite similar to Lemma 2.4 in \cite{ddw} and we omit it here.
  \end{proof}

\begin{lemma}\label{finalue2}
Suppose that $p\neq \frac{n+2s}{n-2s}$. Let $u$ be a solution of (\ref{main}) that is stable outside a ball $B_{R_0}$ and $u_e$ satisfies (\ref{maine}). Then there exists a constant $C>0$ such that
$$ \int_{B_R} y^b u_e^2 \le C R^{n+4-2s \frac{p+1}{p-1}}$$ for any $R>3R_0$.
\end{lemma}
 \begin{proof}
  The extension $u_e$ satisfies
  $$ \bar u(x,y) \le C_{n,s} \int_{\mathbb R^n} u^2(z) \frac{y^{2s}}{(|x-z|^2+y^2)^{\frac{n+2s}{2}}} dz$$
  From this we have
  \begin{eqnarray*}\label{}
 \int_{B_R} y^{3-2s } u_e^2 dx dy &\le& C_{n,s} \int_{|x|\le R, z\in\mathbb R^n}  u_e^2(z) \left(\int_0^R  \frac{y^3}{(|x-z|^2+y^2)^{\frac{n+2s}{2}}}dy \right)
 dz dx
 \\&\le & C_{n,s}  \int_{|x|\le R, z\in\mathbb R^n}  u_e^2(z) \left[\int_0^{R^2} (|x-z|^2+\alpha^2)^{1-\frac{n}{2}-s} d\alpha  - |x-z|^2 \int_0^{R^2} (|x-z|^2+\alpha^2)^{-\frac{n}{2}-s} d\alpha \right]
 \\&=& C_{n,s}  \int_{|x|\le R, z\in\mathbb R^n}  u_e^2(z)  (-2+\frac{n}{2}+s)^{-1} \left[ (|x-z|^2)^{2-\frac{n}{2}-s}-(|x-z|^2+R^2)^{2-\frac{n}{2}-s}    d\alpha \right]
\\&&+C_{n,s}  \int_{|x|\le R, z\in\mathbb R^n}  u_e^2(z) |x-z|^2  (\frac{n}{2}+s-1)^{-1} \left[(|x-z|^2+R^2)^{1-\frac{n}{2}-s}   -   (|x-z|^2)^{2-\frac{n}{2}-s} d\alpha \right]
   \end{eqnarray*}\label{}
 We now split the integral to $|x-z|<2R$ and $|x-z|>2R$. For the case of $|x-z|<2R$ we get

   \begin{eqnarray*}\label{}
&&   \int_{|x|\le R, |x-z|<2R}  u_e^2(z)  (-2+\frac{n}{2}+s)^{-1} \left[ (|x-z|^2)^{2-\frac{n}{2}-s}-(|x-z|^2+R^2)^{2-\frac{n}{2}-s}     \right]
\\&&+   \int_{|x|\le R, |x-z|<2R}  u_e^2(z) |x-z|^2  (\frac{n}{2}+s-1)^{-1} \left[(|x-z|^2+R^2)^{1-\frac{n}{2}-s}   -   (|x-z|^2)^{2-\frac{n}{2}-s}  \right]
\\&\le & C  \int_{|x|\le R, |x-z|<2R}  u_e^2(z)   (|x-z|^2)^{2-\frac{n}{2}-s}
\\&\le& R^{4-2s} \int_{B_{3R}}  u_e^2(z) dz  \le C R^{4-2s} \left(\int_{B_{3R}} |u|^{p+1} \eta_R^2\right)^{2/(p+1)} \left(\int_{B_{3R}}  \eta_R^{-4/(p-1)}   \right)^{(p-1)/(p+1)}
 \\&\le& C R^{4-2s+ n \frac{p-1}{p+1}}  \left(\int_{B_{3R}} u^2(z) \rho_R(z) dz \right)^{2/(p+1)}
\\&\le& C R^{n+4-2s\frac{p+1}{p-1}}
  \end{eqnarray*}\label{}
   Here we have used Lemma \ref{u2} and Lemma \ref{finalu2}. For the case of $|x-z|>2R$ we apply the mean value inequality to get
  \begin{eqnarray*}\label{}
&&   \int_{|x|\le R, |x-z|\ge2R}  u_e^2(z)  (-2+\frac{n}{2}+s)^{-1} \left[ (|x-z|^2)^{2-\frac{n}{2}-s}-(|x-z|^2+R^2)^{2-\frac{n}{2}-s}     \right]
\\&&+   \int_{|x|\le R, |x-z|\ge 2R}  u_e^2(z) |x-z|^2  (\frac{n}{2}+s-1)^{-1} \left[(|x-z|^2+R^2)^{1-\frac{n}{2}-s}   -   (|x-z|^2)^{2-\frac{n}{2}-s}  \right]
\\&\le& C R^4    \int_{|x|\le R, |x-z|\ge 2R}  u_e^2(z)   (|x-z|^2)^{-\frac{n}{2}-s}
\\&\le& C R^4    \int_{|z|\ge R}  u_e^2(z)   \rho dz
 \\&\le& C R^{n+4-2s\frac{p+1}{p-1}}.
     \end{eqnarray*}\label{}
Here we have used Corollary \ref{rhoR}   and Lemma \ref{finalu2}.  This finishes the proof.

  \end{proof}

\begin{lemma}\label{lowest} Let $u$ be a solution of (\ref{main}) that is stable outside a ball $B_{R_0}$ and $u_e$ satisfies (\ref{maine}). Then there exists a positive constant $C$  such that
 \begin{equation}
\int_{B_R\cap\partial\mathbb R_+^{n+1}}  |u_e|^{p+1} + \int_{B_R\cap\mathbb R_+^{n+1}}  y^b |\Delta_b u_e|^2   \le C R^{n-2s \frac{p+1}{p-1}}
  \end{equation}
\end{lemma}
\begin{proof}
This is   a direct consequence of Corollary \ref{ue2} and Lemma \ref{finalue2}.
\end{proof}
 \section{Blow-Down Analysis}
In this section we provide the proof of Theorem \ref{mainthm}.
\\
\\\noindent{\it Proof of Theorem \ref{mainthm}}. Suppose that $u$ is a solution of (\ref{main}) that is stable outside the ball of radius $R_0$ and suppose that $u_e$ is its extension satisfying (\ref{maine}).

Let's first consider the subcritical case, i.e. $1<p\le p_S(n)$. Note that for the subcritical case Lemma implies that $u\in \dot H ^s(\mathbb R^n)\cap L^{p+1}(\mathbb R^n)$. Multiplying (\ref{main}) with $u$ and doing integration, we obtain
\begin{equation}\label{poho1}\int_{\mathbb R^n} |u|^{p+1} = || u||^2_{\dot H^s(\mathbb R^n)}
\end{equation}
in addition multiplying (\ref{main}) with $u^\lambda(x)=u(\lambda x)$ yields
$$ \int_{\mathbb R^n} |u|^{p-1} u^\lambda = \int_{\mathbb R^n} (-\Delta )^{s/2} u (-\Delta)^{s/2} u^\lambda=\lambda^s \int_{\mathbb R^n} w w_\lambda$$ where $w=(-\Delta )^{s/2} u$.  Following ideas provided in \cite{ddww,rs} and the using the change of variable $z=\sqrt\lambda x$ one can get the following Pohozaev identity
$$ -\frac{n}{p+1} \int_{\mathbb R^n} |u|^{p+1}= \frac{2s-n}{2} \int_{\mathbb R^n} w^2+ \frac{d}{d\lambda}\vert_{\lambda=1} \int_{\mathbb R^n} w^{\sqrt \lambda} w^{1/\sqrt\lambda} dz=\frac{2s-n}{2} ||u||^2_{\dot H^s(\mathbb R^n)}$$
This equality together and (\ref{poho1}) proves the theorem for the subcritical case.

We now focus on the supercritical case, i.e. $p> p_S(n)$. We perform the proof in a few steps.
\\
\noindent {\bf Step 1.} $\lim_{\lambda\to\infty} E(u_e,0,\lambda)<\infty$.

From Theorem \ref{mono} $E$ is nondecreasing. So, we only need to show that $E(u_e,0,\lambda)$ is bounded.  Note that
$$E(u_e,0,\lambda) \le \frac{1}{\lambda} \int_\lambda^{2\lambda}  E(u_e,0,t) dt  \le \frac{1}{\lambda^2} \int_\lambda^{2\lambda} \int_t^{t+\lambda} E(u_e,0,\gamma) d\gamma dt $$
From Lemma \ref{lowest} we conclude that
$$ \frac{1}{\lambda^2} \int_\lambda^{2\lambda} \int_t^{t+\lambda}   \gamma^{2s\frac{p+1}{p-1}-n} \left(   \int_{  \mathbb{R}^{n+1}_{+}\cap B_\gamma} \frac{1}{2} y^{3-2s}|\Delta_b u_e|^2 dy dx-  \frac{C_{n,s}}{p+1} \int_{  \partial\mathbb{R}^{n+1}_{+}\cap B_\gamma} u_e^{p+1}  dx \right)  d\gamma dt \le C$$
where $C>0$ is independent from $\lambda$. For the next term in the energy we have
\begin{eqnarray*}\label{}
\frac{1}{\lambda^2} \int_\lambda^{2\lambda} \int_t^{t+\lambda}  \left(\gamma^{-3+2s+\frac{4s}{p-1}-n}  \int_{  \mathbb{R}^{n+1}_{+}\cap \partial B_\gamma} y^{3-2s} u_e^2 dydx\right)d\gamma dt
&\le&\frac{1}{\lambda^2} \int_\lambda^{2\lambda} t^{-3+2s+\frac{4s}{p-1}-n} \int_{B_{t+\lambda}\setminus B_t}   y^{3-2s} u_e^2 dydx  dt
\\&\le& \frac{1}{\lambda^2} \int_\lambda^{2\lambda} t^{-3+2s+\frac{4s}{p-1}-n}\left( \int_{B_{3\lambda}}   y^{3-2s} u_e^2 dydx \right) dt
\\&\le& \lambda^{n+4-2s\frac{p+1}{p-1}}\frac{1}{\lambda^2} \int_\lambda^{2\lambda} t^{-3+2s+\frac{4s}{p-1}-n} dt
\\&\le& C
\end{eqnarray*}
 where $C>0$ is independent from $\lambda$. In the above estimates we have applied  Lemma \ref{finalue2}.  For the next term we have
 \begin{eqnarray*}\label{}
&&\frac{1}{\lambda^2} \int_\lambda^{2\lambda} \int_t^{t+\lambda} \frac{\gamma^3}{2} \frac{d}{d \gamma} \left[ \gamma^{2s-3-n+\frac{4s}{p-1}} \int_{\partial B_\gamma} y^{3-2s} \left( \frac{2s}{p-1} \gamma^{-1} u_e+\frac{\partial u_e}{\partial r} \right)^2  \right] d\gamma dt
\\&=& \frac{1}{2\lambda^2} \int_\lambda^{2\lambda} [(t+\lambda)^{2s-n+\frac{4s}{p-1}} \int_{\partial B_{t+\lambda}} y^{3-2s} \left( \frac{2s}{p-1} (t+\lambda)^{-1} u_e+\frac{\partial u_e}{\partial r} \right)^2
\\&&
- t^{2s-n+\frac{4s}{p-1}}  \int_{\partial B_{\lambda}} y^{3-2s} \left( \frac{2s}{p-1} \gamma^{-1} u_e+\frac{\partial u_e}{\partial r} \right)^2 ] dt
\\&&-\frac{3}{2\lambda^2} \int_\lambda^{2\lambda} \int_t^{t+\lambda}\left[ \gamma^{2s-1-n+\frac{4s}{p-1}}  \int_{\partial B_{\gamma}}  y^{3-2s} \left( \frac{2s}{p-1} \gamma^{-1} u_e+\frac{\partial u_e}{\partial r} \right)^2\right] d\gamma dt
\\&\le& \lambda^{-2+2s -n +\frac{4s}{p-1}} \int_{B_{3\lambda}\setminus B_\lambda} y^{3-2s} \left( \frac{2s}{p-1} \lambda^{-1} u_e+\frac{\partial u_e}{\partial r} \right)^2 \le C
 \end{eqnarray*}
where $C>0$ is independent from $\lambda$. The rest of the terms can be treated similarly.
 \\
 \noindent {\bf Step 2.}  There exists a sequence $\lambda_i\to\infty$ such that $(u_e^{\lambda_i})$ converges weakly in $H^1_{loc}(\mathbb R^n, y^{3-2s} dxdy)$ to a function $u_e^\infty$.

 Note that this is a direct consequence of Lemma \ref{lowest}.
 \\
 \noindent {\bf Step 3.} $u_e^\infty$ is homogeneous.

To prove this claim, apply the scale invariance of $E$, its finiteness and the monotonicity formula; given $R_2>R_1>0$,
  \begin{eqnarray*}\label{}
  0&=& \lim_{i\to\infty} \left(E(u_e,0,R_2\lambda_i)- E(u_e,0,R_1\lambda_i) \right)
\\&=&    \lim_{i\to\infty} \left(E(u_e^{\lambda_i},0,R_2)- E(u_e^{\lambda_i},0,R_1) \right)
\\&\ge&  \liminf_{i\to\infty} \int_{(B_{R_2} \setminus B_{R_1})\cap \mathbb R^{n+1}_+} y^{3-2s}  r^{\frac{4s}{p-1}+2s-2-n}   \left(  \frac{2s}{p-1} r^{-1} u_e^{\lambda_i}+ \frac{\partial u_e^{\lambda_i}}{\partial r}\right)^2 dy dx
\\&\ge&  \int_{(B_{R_2} \setminus B_{R_1})\cap \mathbb R^{n+1}_+} y^{3-2s}  r^{\frac{4s}{p-1}+2s-2-n}   \left(  \frac{2s}{p-1} r^{-1} u_e^{\infty}+ \frac{\partial u_e^{\infty}}{\partial r}\right)^2 dy dx
   \end{eqnarray*}
In the last inequality we have used the weak convergence of $(u_e^{\lambda_i})$ to $u_e^\infty$ in $H^1_{loc}(\mathbb R^n,y^{3-2s} dydx)$. This implies
$$  \frac{2s}{p-1} r^{-1} u_e^{\infty}+ \frac{\partial u_e^{\infty}}{\partial r}=0 \ \ \text{a.e. \ \ in} \ \ \mathbb R_+^{n+1}.$$
  Therefore, $u_e^\infty$ is homogeneous.
  \\
 \noindent {\bf Step 4.} $u_e^\infty=0$.

This is a direct consequence of Theorem \ref{homog}.
    \\
 \noindent {\bf Step 5.} $(u_e^{\lambda_i})$ converges strongly to zero in $H^1(B_R\setminus B_\epsilon, y^{3-2s} dydx)$ and $(u_e^{\lambda_i})$ converges strongly to zero in $L^{p+1}(B_R\setminus B_\epsilon)$ for all $R>\epsilon>0$.
\\\noindent {\bf Step 6.} $u_e\equiv0$.
\begin{eqnarray*}
I( u_e,\lambda) &=& I( u_e^\lambda,1)
\\&=&
\frac{1}{2}\int_{\r\cap B_{1}} y^{3-2s}\vert\Delta_b u_e^\lambda\vert^2  dxdy -\frac{\kappa_{s}}{p+1} \int_{\br\cap B_{1}} \vert u_e^\lambda\vert^{p+1}dx
\\&=&\frac{1}{2} \int_{\r\cap B_{\epsilon}} y^{3-2s} \vert \Delta_b u_e^\lambda\vert^2 dx dy - \frac{\kappa_{s}}{p+1} \int_{\br\cap B_{\epsilon}}  \vert u_e^\lambda\vert^{p+1}dx
\\&&+\frac{1}{2} \int_{\r\cap B_{1}\setminus B_{\epsilon}} y^{3-2s} \vert\Delta_b u_e^\lambda\vert^2 dx dy - \frac{\kappa_{s}}{p+1} \int_{\br\cap B_{1}\setminus B_{\epsilon}}  \vert u_e^\lambda\vert^{p+1}dx
\\&=&\eps^{n-\frac{2s(p+1)}{p-1}} I(u_e,\lambda\eps) +\frac{1}{2} \int_{\r\cap B_{1}\setminus B_{\epsilon}} y^{3-2s} \vert\Delta_b u_e^\lambda\vert^2 dx dy -\frac{\kappa_{s}}{p+1} \int_{\br\cap B_{1}\setminus B_{\epsilon}} \vert  u_e^\lambda\vert^{p+1}dx\\
&\le& C\eps^{n-\frac{2s(p+1)}{p-1}} + \frac{1}{2} \int_{\r\cap B_{1}\setminus B_{\epsilon}} y^{3-2s} \vert\Delta u_e^\lambda\vert^2 dx dy - \frac{\kappa_{s}}{p+1} \int_{\br\cap B_{1}\setminus B_{\epsilon}}   \vert  u_e^\lambda\vert^{p+1}dx
\end{eqnarray*}
Letting $\lambda\to+\infty$ and then $\eps\to0$, we deduce that
$
\lim_{\lambda\to+\infty} I( u_e,\lambda) =0.
$
Using the monotonicity of $E$,
\begin{equation}
E(u_e,\lambda) \le \frac1\lambda\int_{\lambda}^{2\lambda}E(t)\;dt\le \sup_{[\lambda,2\lambda]}I + C\lambda^{-n-1+\frac{2s(p+1)}{p-1}}\int_{B_{2\lambda}\setminus B_{\lambda}}u_e^2\\
\end{equation}
and so
$
\lim_{\lambda\to+\infty}E(u_e,\lambda) =0.$ Since $u$ is smooth, we also have $E(u_e,0)=0$. Since $E$ is monotone, $E\equiv 0$ and so $\bar u$ must be homogeneous, a contradiction unless $u_e\equiv0$.

\begin{remark} Note that we expect that when (\ref{conditionp}) does not hold that is when 
\begin{equation}\label{ffff}
p \frac{\Gamma(\frac{n}{2}-\frac{s}{p-1}) \Gamma(s+\frac{s}{p-1})}{\Gamma(\frac{s}{p-1}) \Gamma(\frac{n-2s}{2}-\frac{s}{p-1})} \le \frac{ \Gamma(\frac{n+2s}{4})^2 }{\Gamma(\frac{n-2s}{4})^2}
\end{equation}
there exist radial entire stable solutions.   The method of construction of such solutions is the one that is applied in \cite{ddw} and references therein.   More precisely, one needs to mimic the standard proof for the existence of a  minimal solution that is axially symmetric for the associated problem on bounded domains.  Then applying the truncation method and the moving plane method one can show that the minimal solution is bounded and radially decreasing.  From elliptic estimates and some classical convexity arguments  the minimal solution would converge to the singular solution that is stable.  This implies that (\ref{ffff}) should hold.  Finally using the singular solution and the minimal solution one can construct a radial, bounded and smooth solution via rescaling arguments.  

\end{remark}


\begin{thebibliography}{99}

\bibitem{cgs}  L. Caffarelli, B. Gidas, J. Spruck, \emph{Asymptotic symmetry and local behavior
of semilinear elliptic equations with critical Sobolev growth}, Comm. Pure Appl. Math. 42
(1989), no. 3, 271-297

    \bibitem{cs} L. Caffarelli and L. Silvestre, \emph{An extension problem related to the fractional Laplacian}, Comm. Partial Differential Equations 32 (2007), no. 7-9, 1245-1260.

    \bibitem{cc} J. Case, Sun-Yung Alice Chang, \emph{On fractional GJMS operators}, preprint http://arxiv.org/abs/1406.1846



 \bibitem{cg}   Sun-Yung Alice Chang and Maria del Mar Gonzalez, \emph{Fractional Laplacian in conformal geometry},  Advances in Mathematics  226 (2011), no. 2, 1410-1432.


\bibitem{clo} W. Chen, C. Li, B. Ou, \emph{Classification of solutions for an integral
equation}, Comm. Pure Appl. Math. 59 (2006), no. 3, 330-343.

\bibitem{cczy} W. Chen, X. Cui, R. Zhuo, Z. Yuan, A Liouville theorem for the fractional laplacian, arxiv:1401.7402v1.


    \bibitem{ddw} J. Davila, L. Dupaigne, J. Wei, \emph{On the fractional Lane-Emden equation}, preprint.

    \bibitem{ddww} J. Davila, L. Dupaigne and K. Wang, J. Wei, \emph{A Monotonicity Formula and a Liouville-type Theorem for a Fourth Order Supercritical Problem}, Advances in Mathematics 258 (2014), 240-285.


  \bibitem{fall}  M. Fall, \emph{Semilinear elliptic equations for the fractional Laplacian
with Hardy potential}, preprint. http://arxiv.org/pdf/1109.5530v4.pdf

\bibitem{f} A. Farina;
 \emph{On the classification of solutions of the Lane-Emden equation on unbounded domains of $\Bbb R^N$}, J. Math. Pures Appl. (9) 87 (2007), no. 5, 537-561.

\bibitem{gg} F. Gazzola, H. C. Grunau, \emph{Radial entire solutions for supercritical biharmonic equations},  Math. Annal.  334 (2006) 905-936.

\bibitem{gs}  B. Gidas, J. Spruck, \emph{{A priori bounds for positive solutions of nonlinear elliptic equations}, Comm. Partial Differential Equations}  6 (1981) 883-901.

\bibitem{h} Ira W. Herbst, \emph{Spectral theory of the operator $(p^2 +m^2)^{1/2} - Ze^2/r$}, Comm. Math. Phys. 53 (1977), no. 3, 285-294.


\bibitem{li} Y. Li, \emph{Remark on some conformally invariant integral equations: the method of moving spheres}, J. Eur. Math. Soc. (JEMS) 6 (2004), no. 2, 153-180.

       \bibitem{lin} C. S. Lin, \emph{A classification of solutions of a conformally invariant fourth order equation in $\mathbb{R}^{N}$,} Comment. Math. Helv. 73 (1998) 206-231.

  \bibitem{jl}     D. D. Joseph, T. S. Lundgren, \emph{Quasilinear Dirichlet problems driven by positive sources}, Arch. Rational Mech. Anal. 49 (1972/73) 241- 269.

  \bibitem{pak} F. Pacard, \emph{A note on the regularity of weak solutions of $\Delta u=u^\alpha$ in $\mathbb R^n$}, Houston J. Math. 18 (1992) no 4, 621-632.

 \bibitem{rs} Ros-Oton, J. Serra, \emph{Local integration by parts and Pohozaev identities for higher order fractional Laplacians}, preprint  http://arxiv.org/abs/1406.1107

\bibitem{wx} J.  Wei, X. Xu; \emph{Classification of solutions of higher order conformally invariant equations}, Math. Ann. 313 (1999), no. 2, 207-228.

\bibitem{ya} D. Yafaev, \emph{Sharp Constants in the Hardy-Rellich Inequalities}, Journal of Functional Analysis 168, (1999) 121-144.


\bibitem{yang} R. Yang, \emph{On higher order extensions for the fractional Laplacian}, preprint. http://arxiv.org/pdf/1302.4413v1.pdf


\end{thebibliography}
\end{document}